\documentclass[reqno,11pt]{amsart}
\usepackage{amsmath,amssymb,mathrsfs,amsthm,amsfonts}
\usepackage[inline]{enumitem} 
\usepackage[usenames,dvipsnames]{xcolor}
\usepackage{hyperref}
\usepackage{comment}
\usepackage{tikz}
\usetikzlibrary{matrix,graphs,arrows,positioning, shapes, calc,decorations.markings,decorations.pathmorphing,shapes.symbols,angles}

\hypersetup{%
	colorlinks=true, linkcolor=blue,
	citecolor=ForestGreen
}
\usepackage[paper=letterpaper,margin=1in]{geometry}

\newtheorem{theorem}{Theorem}[section]
\newtheorem{lemma}[theorem]{Lemma}
\newtheorem{proposition}[theorem]{Proposition}
\newtheorem{corollary}[theorem]{Corollary}
\theoremstyle{definition}
\newtheorem{definition}[theorem]{Definition}
\newtheorem{conjecture}[theorem]{Conjecture}
\newtheorem{remark}[theorem]{Remark}
\numberwithin{equation}{section}
\allowdisplaybreaks
\usepackage{acronym}

\newcommand{\BR}{\operatorname{BR}}
\newcommand{\e}{\varepsilon}

\renewcommand{\L}{\mathcal{L}}

\newcommand{\EE}{\mathsf{E}}

\renewcommand{\AA}{\mathsf{A}}
\newcommand{\tw}{\operatorname{TW}_{\operatorname{GUE}}}

\newcommand{\DD}{\mathsf{D}}

\newcommand{\h}{\mathfrak{h}}

\newcommand{\sh}{\mathfrak{h}}
\newcommand{\Y}{\mathsf{Y}}
\renewcommand{\P}{\mathbb{P}}

\newcommand{\Z}{\mathbb{Z}}

\newcommand{\af}{\mathrm{A}}

\newcommand{\R}{\mathbb{R}}
\renewcommand{\P}{\mathbb{P}}

\newcommand{\Con}{\mathrm{C}}
\newcommand{\csq}{\Con_{\operatorname{sq}}}
\numberwithin{equation}{section}

\renewcommand{\L}{\mathcal{L}}

\acrodef{KPZ}{Kardar--Parisi--Zhang}
\acrodef{SHE}{Stochastic Heat Equation}
\acrodef{LDP}{Large Deviation Principle}

\renewcommand{\Pr}{\mathbb{P}}	
\renewcommand{\d}{\mathrm{d}}
\newcommand{\dt}{\kappa}	
\newcommand{\ind}{\mathbf{1}}	
 
\newcommand{\bet}{\sigma} 

\newcommand{\nt}[1]{\nabla_{#1}}
\newcommand{\ft}[1]{\Lambda_{#1}}
\newcommand{\B}{\mathfrak{B}}
\newcommand{\bm}{\mathfrak{B}}

\newcommand{\norm}[1]{\Vert#1\Vert}

\newcommand{\stp}{\mathfrak{p}}
\newcommand{\ltp}{\mathfrak{q}}
\newcommand{\st}{ST-class($\stp$)}
\newcommand{\lt}{LT-class($\ltp$)}

\newcommand{\g}{\mathfrak{g}}
\newcommand{\m}{\mathsf}

\newcommand{\calA}{\mathcal{A}}

\newcommand{\calD}{\mathcal{D}}

\newcommand{\cl}{\mathcal{L}}
\newcommand{\ck}{\mathcal{K}}

\newcommand{\til}{\widetilde}

\renewcommand{\P}{\mathbb{P}}

\newcommand{\se}[1]{\Con\exp\big(-\tfrac1{\Con}s^{{#1}}\big)}
\newcommand{\te}[1]{\Con\exp\big(-\tfrac1{\Con}t^{{#1}}\big)}

\newcommand{\red}[1]{\textcolor{black}{#1}}
\usepackage{graphicx}

\title[Long and short time LIL for KPZ fixed point]{Long and short time laws of Iterated Logarithms for the KPZ fixed point}

\author[S.\ Das]{Sayan Das}
\address{S.\ Das,
	Department of Mathematics, Columbia University,
	\newline\hphantom{\quad \ \ S. Das}
	2990 Broadway, New York, NY 10027, USA
}
\email{sayan.das@columbia.edu}
\author[P.\ Ghosal]{Promit Ghosal}
\address{P.\ Ghosal,
	Department of Mathematics, Massachusetts Institute of Technology,
	\newline\hphantom{\quad \ \ P. Ghosal}
	182 Memorial Drive, Cambridge, MA 02139, USA
}
\email{promit@mit.edu}
\author[Y.\ Lin]{Yier Lin}
\address{Y.\ Lin,
	Department of Statistics, University of Chicago,
	\newline\hphantom{\quad \ \ Y. Lin}
	5747 S Ellis Ave, Chicago, IL 60637, USA
}
\email{ylin10@uchicago.edu}

\begin{document}
	\begin{abstract} We consider the KPZ fixed point  starting from a general class of initial data. In this article, we study the growth of the large peaks of the KPZ fixed point at a spatial point $0$ when time $t$ goes to $\infty$ and when $t$ approaches $1$. We prove that for a very broad class of initial data,
		as $t\to \infty$, the limsup of the KPZ fixed point height function when scaled by $t^{1/3}(\log\log t)^{2/3}$ almost surely equals a constant. The value of the constant is $(3/4)^{2/3}$ or $(3/2)^{2/3}$ depending on the initial data being non-random or Brownian respectively.
		Furthermore, we show that the increments of the KPZ fixed point near $t=1$ admits a short time law of iterated logarithm. More precisely, as the time increments $\Delta t :=t-1$ goes down to $0$, for a large class of initial data including the Brownian data initial data, we show that limsup of the height increments 
		the KPZ fixed point near time $1$ when scaled by $(\Delta t)^{1/3}(\log\log (\Delta t)^{-1})^{2/3}$ almost surely equals  $(3/2)^{2/3}$.

	\end{abstract}
	
	
	\maketitle
	
	\section{Introduction}
	
	\subsection{Background}	The Kardar-Parisi-Zhang (KPZ) universality class consists of a broad family of random growth models that are believed to exhibit certain common features, such as universal scaling exponents and limiting distributions. Several models such as asymmetric simple exclusion processes, last passage percolation, directed polymers in random environment, driven lattice gas models, KPZ equation are believed to lie in this class. All these models are conjectured to have an universal limiting structure, namely the \textit{KPZ fixed point} under the so called `KPZ' scaling.     
	
	\smallskip
	
	Over the last four decades, an immense progress has been achieved in understanding a large collection of models in the KPZ universality class (see \cite{quastel2011introduction,corwin2012kardar,quastel2015one,HHNT,Tak18}) by intermingling the ideas from stochastic PDEs, random matrix theory, representation theory of the quantum groups and various other research directions. The universal scaling limit, the KPZ fixed point, itself was rigorously first constructed in \cite{MQR16} as a scaling limit of TASEP by specifying the transition probabilities. An alternative description of the KPZ fixed point as a variational formula involving \textit{space-time Airy sheet} was conjectured in \cite{co}. This space-time Airy sheet, later known as the \textit{directed landscape} (modulo a parabolic shift), was derived rigorously as a scaling limit of Brownian last-passage percolation in \cite{DOV18}. The connection between the two objects was obtained in \cite{NQR20}, where the authors proved the conjectural variational formulation of the KPZ fixed point, providing a coupling of all initial data on the same probability space based on the directed landscape.  
	
	\smallskip
	
	In this paper, we introduce the KPZ fixed point using the variational formulation which involves the directed landscape. We introduce the latter below.

	\begin{definition}[Directed Landscape] \label{def:dl} Let $\R^4_{\uparrow}:=\{(x,s;y,t)\in \R^4: s<t\}$. The directed landscape is a random continuous function $\cl :\R^4_{\uparrow}\to \R$ satisfying the metric composition law
		\begin{align}\label{metlaw}
			\cl(x,r;y,t)=\sup_{z\in \R} \left\{\cl(x,r;z,s)+\cl(z,s;y,t) \right\}, \mbox{ for all } (x,r;y,t)\in \R^4_{\uparrow}, \ s\in (r,t),
		\end{align}
		with the property that $\cl(\cdot,t_i;\cdot,t_i+s_i^3)$ are independent for any set of disjoint intervals $(t_i,t_i+s_i^3)$ (see Definition 1.2 in \cite{DOV18} for definition of Airy sheet). and as a function in $x,y$, $\cl(x,t;y,t+s^3)\stackrel{d}{=}s\cdot\mathcal{S}(x/s^2,y/s^2)$, where $\mathcal{S}(\cdot,\cdot)$ is a parabolic Airy sheet (see Definition 1.2 in \cite{DOV18} for definition of parabolic Airy sheet). The marginal of parabolic Airy sheet satisfies $\mathcal{S}(0,x)\stackrel{d}{=}\mathcal{A}(x)-x^2$ where $\calA(x)$ is the stationary $\operatorname{Airy}_2$ process constructed in \cite{ps02}.
	\end{definition}
	\begin{definition}[KPZ fixed Point]
		Given a directed landscape $\cl$ and an independent initial data $\h_0$, the KPZ fixed point $\h_t(x)$ is given by
		\begin{align}\label{def:kpzfp}
			\h_t(x):=\sup_{z\in \R}\left\{\h_0(z)+\cl(z,0;x,t)\right\}.
		\end{align}
	\end{definition}
	Initial data for the KPZ fixed point plays an important role in the precise nature of the fluctuations of the height function. Owing to exact solvability in some prelimiting KPZ models, the KPZ fixed point marginals for the three fundamental initial data: \textit{narrow wedge, flat, and Brownian}, has been well studied in the literature.
	
	\noindent $\bullet$ When $\h_0=-\infty\ind_{x\neq0}$, \textit{the narrow wedge initial data}, the distribution of the KPZ fixed point height function at any finite time is given by $\calA(\cdot)$ minus a parabola where $\calA(\cdot)$ is the stationary $\operatorname{Airy}_2$ process \cite{ps02} whose one point distribution are given by Tracy-Widom GUE distribution \cite{tw}.
	
	\noindent $\bullet$ For flat initial data, $\h_0\equiv 0$, $\h_1(\cdot)$ is given by the $\operatorname{Airy}_1$ process discovered by Sasamoto \cite{sasamoto2005spatial}. The one-point distribution of $\operatorname{Airy}_1$ process are related to Tracy-Widom GOE distribution \cite{tw2}.
	
	\noindent $\bullet$ When $\h_0(x)=\bm(x)$, a two sided Brownian motion with diffusion coefficient $2$, the	KPZ fixed point becomes stationary in the sense that $\h_t(x)-\h_t(0) \stackrel{d}{=}\bm(x)$ for each $t$. The distribution of $\h_1(\cdot)$ in this case is given by $\operatorname{Airy}_{\operatorname{stat}}$ process obtained by \cite{bbp}. It has Baik-Rains distribution as its one-point marginals \cite{br,chhita2018limit}. 
	
	In this paper we consider the KPZ fixed point started from a suitable class of initial data which includes the above three fundamental initial data.

	\subsection{Main results} 	
	We investigate the evolution of the heights of the large peaks of the KPZ fixed point in global and local scale through the lens of law of iterated logarithms (LIL). More precisely, we study LIL for the height function of the KPZ fixed point $\h_t(0)$ at spatial location $x=0$ as $t\uparrow \infty$ (long-time LIL), as well as the temporal increment of the KPZ fixed point near $1$, i.e., $\h_t(0)-\h_1(0)$ as $t\downarrow 1$ (short-time LIL). 
	We first state our long-time LIL result below.

	\begin{theorem}[Long-time LIL for non-random data] \label{thm:longLILa}    Let $\h_0$ be either narrow wedge, or a non-random Borel-measurable function with at most $\sqrt{x}$ growth, i.e., there exists a constant $\csq>0$ such that
		\begin{align}\label{eq:sq-growth}
			\h_0(x) \le \csq\sqrt{1+|x|}, \ \forall \ x\in \R.
		\end{align}
		Consider the KPZ fixed point $\h$ started from $\h_0$. We have
		\begin{align}\label{eq:main3}
			\limsup_{t\to \infty} \frac{\h_t(0)}{t^{\frac13}(\log\log t)^{2/3}} \stackrel{a.s.}{=} \left(\frac34\right)^{\frac23}.
		\end{align}	
	\end{theorem}	
	
	The above result identifies the limiting height of the peaks of the KPZ fixed point. Note that constant in \eqref{eq:main3} is same for all initial data satisfying \eqref{eq:sq-growth}. This naturally leads to the following question:
	
	\smallskip
	
	\noindent $\bullet$ \textit{Where are the precise constant $(3/4)^{2/3}$ and $(\log\log t)^{2/3}$ behavior coming from?}

	As in other cases like Brownian motion or the Kardar-Parisi-Zhang equation \cite{dg21}, the value of the constant in the law of iterated logarithms of the KPZ fixed point is governed by the tail behavior of the one point distribution. As it turns out, for the initial data considered above, the (upper)-tail behavior of $t^{-1/3}\h_t(0)$ is similar to the (upper)-tail behavior of Tracy-Widom GUE distribution. If $X$ is a Tracy-Widom GUE distributed random variable, the upper tail, $\Pr(X\ge s)$, roughly behaves like $e^{-4s^{3/2}/3}$ for large $s$. The $3/2$ exponent in the tail induces a $(\log\log t)^{2/3}$ type behavior in the law of iterated logarithm whereas the constant $(3/4)^{2/3}$ comes from the leading coefficient in the exponent.
	
	\noindent $\bullet$ \textit{Why do we need condition \eqref{eq:sq-growth} on the initial data?}	
	
	As depicted in Theorem~\ref{thm:longLILa}, the initial data can atmost grow as $\sqrt{x}$. We now explain the reason behind such a choice of initial data. It is well known that the directed landscape $\cl(z,0;0,t)$ intuitively behaves like $-z^2/t$ (see Section \ref{sec:dl-decay}). Thus following \eqref{def:kpzfp} we heuristically expect $\h_t(0)$ to grow like as 
	\begin{align*}
		\sup_{z\in \R} \{\h_0(z)-\tfrac{z^2}{t}\}= t^{1/3}\sup_{x\in \R} \{t^{-1/3}\h_0(t^{2/3}x)-x^2\}.
	\end{align*}
	Note that $\h_t(0)$ is of the order $t^{1/3}$ as long as $\h_0(x) \le \csq\sqrt{1+|x|}$. For initial data with higher growth, one needs to subtract a non-trivial growing mean term from $\h_t(0)$ to see $t^{1/3}$ order fluctuations. Up to such centering, we believe that similar result as in \eqref{eq:main3} would hold. However, we defer this case to some future work. Theorem \ref{thm:longLILa} leaves out the Brownian initial data. We now state the corresponding result for Brownian data separately.
	
	\begin{theorem}[Long-time LIL for Brownian data] \label{thm:longLILb} Consider the KPZ fixed point $\h$ started from $\h_0(x)=\bm(x)$ where $\bm$ is a two-sided Brownian motion with diffusion coefficient $2$. Almost surely we have
		\begin{align}\label{eq:main4}
			\limsup_{t\to \infty} \frac{\h_t(0)}{t^{\frac13}(\log\log t)^{2/3}} = \left(\frac32\right)^{\frac23}.
		\end{align}	
	\end{theorem}	
	\medskip
	
	Note that although $(\log\log t)^{2/3}$ scaling above is same as Theorem \ref{thm:longLILa}, the limiting constant in \eqref{eq:main4} is different compared to \eqref{eq:main3}. This is due to the fact the distribution of $t^{-1/3}\h_t(0)$ for Brownian data is given by the Baik-Rains distribution which has $e^{-2s^{3/2}/3}$ upper tail decay \cite{ferrari2021upper} in comparison to the tail decay $e^{-4s^{3/2}/3}$ for the Tracy-Widom GUE distribution. This leads to the above $(3/2)^{2/3}$ constant.

	Theorems \ref{thm:longLILa} and \ref{thm:longLILb} provide precise long-time LIL for the KPZ fixed point for a large class of data. However, it leaves out the case of liminf type law of iterated logarithm which corresponds to the limiting depth of the valleys of the KPZ fixed point. Such liminf result is related to how the lower tail probability $\Pr(t^{-1/3}\h_t(0)\le -s)$ decays with $s$. For a large class of initial data, the lower tail probability decays like  $e^{-\gamma s^{3}}$ for large $s$. Here $\gamma:=\frac{1}{12}$ for narrow wedge and Brownian data and $\gamma:=\frac16$ for flat initial data (see Example 1.16 in \cite{qr19}). Thus we expect 
	\begin{align}\label{eq:liminf}
		\liminf_{t\to \infty} \frac{\h_t(0)}{t^{1/3}(\log\log t)^{1/3}}\stackrel{a.s.}{=}-\gamma^{-1/3}.
	\end{align}
	However, it is not known what would be the constant $\gamma$ in the lower tail for all initial data satisfying \ref{thm:longLILa}. As a result, it is less clear what should be the limiting constant for the liminf case. Once the value of the constant is determined, the proof of \eqref{eq:liminf} will be very similar to that of Theorem~\ref{thm:longLILa}.

	\smallskip

	We next turn towards our short-time LIL result that identifies local temporal growth behavior.
	
	\begin{theorem}[Short-time LIL] \label{thm:main} Let $\h_0$ be the narrow wedge initial data or two-sided Brownian motion with diffusion coefficient $2$ or a non-random Borel measurable function with at most parabolic growth of size less than $-x^2$, i.e., there exist constants $\af\in \R$ and $\dt>0$ such that
		\begin{align}\label{eq:condd}
			\h_0(x)\le \af+(1-\dt)x^2, \ \forall \ x\in \R.
		\end{align}
		Consider the KPZ fixed point $\h$ started from the initial data $\h_0$. Almost surely we have
		\begin{align}\label{eq:main2}
			\limsup_{\e\downarrow 0} \frac{\h_{1+\e}(0)-\h_1(0)}{\e^{\frac13}(\log\log\e^{-1})^{2/3}} = \left(\frac32\right)^{\frac23}.
		\end{align}	
	\end{theorem} 
	
	\smallskip
	We see that local growth of the KPZ fixed point in the temporal direction is of the order $\e^{1/3}(\log\log \e^{-1})^{2/3}$. Furthermore, the precise constant $(3/2)^{2/3}$ is universal in the sense that it does not depend on the initial data. We now briefly explain why this is the case. Indeed note that due to the metric composition law in \eqref{metlaw} and the variational formula in \eqref{def:kpzfp} we have
	\begin{align} \label{eq:arek1}
		\nabla \h_{\e}:= \frac{\h_{1+\e}(0)-\h_1(0)}{\e^{1/3}} = \sup_{x\in \R} \left[\frac{\h_1(\e^{2/3}x)-\h_1(0)}{\e^{1/3}}+\frac{\cl(\e^{2/3}x,1;0,1+\e)}{\e^{1/3}}\right].
	\end{align}
	By the scale invariance properties of the directed landscape, $\varepsilon^{-\frac{1}{3}}\mathcal{L}(\varepsilon^{\frac{2}{3}}x, 1; 0, 1+\varepsilon)$ is equal in distribution to $\calA(x)-x^2$ where $\calA(\cdot)$ is the stationary $\operatorname{Airy}_2$ process. On the other hand, by local Brownianity of the KPZ fixed point (see Theorem 4.14 in \cite{MQR16}), for any initial data of Theorem~\ref{thm:main}, we have
	\begin{align*}
		\frac{\h_1(\e^{2/3}x)-\h_1(0)}{\e^{1/3}} \stackrel{d}{\to} \B(x), \qquad \text{as } \e \downarrow 0
	\end{align*}
	where $\B(x)$ is a two sided Brownian motion with diffusion coefficient $2$. Thus, as $\e \downarrow 0$, we expect $\nabla \h_{\e}$ converge to the Baik-Rains distribution \cite{br} which is defined as $\sup_{x\in \R} (\B(x)+\calA(x)-x^2)$. The constant $(3/2)^{2/3}$ appearing in \eqref{eq:main2} can then be anticipated from the sharp upper tail asymptotics of the Baik-Rains distribution \cite{ferrari2021upper}.  
	
	\smallskip
	
	Note that Theorem \ref{thm:main} includes a larger class of functional data compared to Theorem \ref{thm:longLILa}. This is partly because  we are interested in the local growth behavior of the KPZ fixed point around $t=1$. As we explain below, studying local growth of the KPZ fixed point does not require a non-trivial centering of $\h_1(0)$ for a larger class of initial data as opposed to the case in Theorem~\ref{thm:longLILa}. To see this, note that the directed landscape $\cl(z,0;0,1)$ decays like $-z^2$. Thus following \eqref{def:kpzfp} we  expect $\h_1(0)$ to be centered around
	\begin{align*}
		\sup_{z\in \R} \{\h_0(z)-z^2\}.
	\end{align*}
	The last expression is finite whenever $\h_0(x) \le \af+(1-\dt)x^2$. This justifies the condition in \eqref{eq:condd}. 
	
	Like as the limsup result of Theorem~\ref{thm:main}, the height increments of the KPZ fixed point near $t=1$ also admit a liminf LIL as the time increments converges to $0$. More precisely, for the KPZ fixed point started from any initial data considered in Theorem~\ref{thm:main}, we expect 
	$$\liminf_{\e\downarrow 0} \frac{\h_{1+\e}(0)-\h_1(0)}{\e^{\frac13}(\log\log\e^{-1})^{1/3}} = -12^{\frac13}.$$
	The proof of the above fact can be completed using similar ideas as in the proof of Theorem~\ref{thm:main}. We defer showing this result to a future work.

	 Theorem~\ref{thm:main} studies the local growth of the KPZ fixed point along temporal direction near the time $1$. However the same result holds for any other fixed time $t$ (assuming $\h_0(x)\le \af+(1-\dt){x^2}/{t}$ instead of \eqref{eq:condd} for functional data) and the proof follows in the same way. It is worthwhile to ask if the same result hold for all $t\in [c,d]$ uniformly. In other words, how often does the short time LIL for the KPZ fixed point fails to hold on an interval. Similar questions were studied for the Brownian motion in \cite[Theorem~2]{OreyTaylor74} where the authors had showed the Hausdorff dimension of the set of time points where the LIL fails is almost surely equal to $1$. In fact, their result is much deeper. Taking inspiration from their result and drawing on the analogy between the short time LILs of the KPZ fixed point and the Brownian motion, we make the following conjecture. 
	
	\begin{conjecture}
	Continuing with the notation and assumption of Theorem~\ref{thm:main}, define the set 
	\begin{align*}
	    \mathbf{E}(\alpha) = \Big\{s\in [1,2] \mid \limsup_{t\downarrow s} \frac{\mathfrak{h}_t(0) - \mathfrak{h}_s(0)}{(t-s)^{1/3}(\log (t-s)^{-1})^{2/3}}\geq \Big(\tfrac{3}{2}\Big)^{\frac{2}{3}}\alpha\Big\}
	\end{align*}
	where $\alpha\in [0,1]$. Then the Hausdorff dimension of the set $\mathbf{E}(\alpha)$ is equal to $1-\alpha^{\frac{3}{2}}$.
	\end{conjecture}
	
	.     
	
	\subsection{Proof Ideas}\label{sec1.1}  We now briefly describe the proof idea of our main results. We first explain the details for the short-time LIL case. The spirit of the arguments for the long time LIL is similar and will be touched upon later in this section.

	\smallskip

	Let us recall the notation $\nabla \h_\e$ from \eqref{eq:arek1}. As alluded in the discussion after \eqref{eq:arek1}, as $\e\downarrow 0$, we expect $\nabla \h_\e$ converges to the Baik-Rains distribution. Although we have precise tail estimates of the Baik-Rains distribution, to conclude the law of iterated logarithm type result, we need to produce sharp uniform tail estimates for $\nabla \h_{\e}$. One of the main contributions of this paper is to establish such estimates.

	\begin{proposition}[Simpler version of Proposition \ref{prop:stincrementtail}]\label{sims} Given any $\delta>0$, there exist $s_0(\delta),\e_0(\delta)>0$, such that for all $s>s_0$ large enough and for all $\e\in (0,\e_0)$ small enough we have 
		\begin{align}\label{eq:arek}
			\Pr(\nabla\h_{\e}\ge s) \le \exp\left(-(\tfrac23-\delta)s^{3/2}\right).
		\end{align}
	\end{proposition}

	Let us briefly explain the proof idea behind the above result. Towards this end, let us define
	\begin{align} \label{eq:arek2}
		\nabla \h_{\e}^I:=  \sup_{x\in I} \left[\frac{\h_1(\e^{2/3}x)-\h_1(0)}{\e^{1/3}}+\calA(x)-x^2\right],
	\end{align}
	where $\calA$ is independent of $\h_1$.	The proof of the above proposition proceeds in two steps:
	
	\begin{enumerate}[label=(\roman*)]
		\item \label{item1} There exists an absolute constant $b>0$ such that
		$\Pr(\nabla \h_\e\neq \nabla \h_{\e}^{[-b\sqrt{s},b\sqrt{s}]}) \le \exp(-\frac23s^{3/2})$.
		\item \label{item2} Proving Proposition \ref{sims} with $\nabla \h_\e$ replaced by $\nabla \h_{\e}^{[-b\sqrt{s},b\sqrt{s}]}$. 
	\end{enumerate}
	
	We first describe how to show item \ref{item2}. For instructive purposes, here we explain how to derive it for $\nabla\h_{\e}^{[0,b\sqrt{s}]}$. Set $a:=\e^{2/3}b\sqrt{s}$. Let us denote
	\begin{align*}
		Z(\pm a):=\underset{x\in\R}{\operatorname{argmax}} (\h_0(x)+\cl(x,0;\pm a,1)), \quad  Z^{\mu}(\pm a):=\underset{x\in \R}{\operatorname{argmax}} (\B(x)+\mu x+\cl(x,0;\pm a,1)),
	\end{align*}
	where $\B$ is an independent two sided Brownian motion with diffusion coefficient $2$. Let us assume for simplicity the above argmax are uniquely defined almost surely. The key idea behind our proof is a Brownian replacement principle that will allow us to \red{compare the KPZ fixed point starting from a general initial data with that starting from} the Brownian initial data on certain events. Indeed,	owing to the geometric properties of the KPZ fixed point, \cite{pim} argues for any $\mu>0$ on the event $\EE^{\mu}(a):=\{Z(a)\le Z^{\mu}(-a), Z(-a)\ge Z^{-\mu}(a)\}$ one has
	\begin{align*}
		&\h_1^{-\mu} (x) - \h_1^{-\mu} (0) \leq \h_1(x)-\h_1(0) \le \h_1^{\mu}(x)-\h_1^{\mu}(0), \mbox{ for }x\in [0,a],\\
		&\h_1^{\mu} (x) - \h_1^{\mu} (0) \leq \h_1(x)-\h_1(0) \le \h_1^{-\mu}(x)-\h_1^{-\mu}(0), \mbox{ for }x\in [-a, 0]
	\end{align*}
	where we set $\h_1^{\mu}(x):=\sup_{z\in \R} (\B(z)+\mu x+\cl(z,0;x,1))$ for all $x\in \R$. 
	
	\medskip
	
	We call the above trick by the name \emph{Brownian replacement trick} which was first discovered in the present form by Leandro Pimentel \cite{pim14}. \red{It is essentially a property of the basic coupling in last passage percolation models}.  On $\EE^{\mu}(a)$, we can then transfer our analysis to the KPZ fixed point started with Brownian motion with drift as initial data. The advantage of working with this initial data is that such data is known to be stationary for the KPZ fixed point (see Theorem 4.5 \cite{MQR16}). In particular we have $\h_1^{\mu}(x)-\h_1^{\mu}(0) \stackrel{d}{=} \B(x)+\mu x$. Under the diffusive scaling appearing in \eqref{eq:arek1},   $\e^{-1/3}(\h_1^{\mu}(\e^{2/3} x)-\h_1^{\mu}(0))$ is then equal in distribution to $\bm(x)+\e^{1/3}\mu x$. As $x$ varies in $[0,b\sqrt{s}]$, the drift term is bounded by $\e^{1/3}\mu b\sqrt{s}$. Thus on $\EE^{\mu}(a)$,
	$$\nabla \h_{\e}^{[0,b\sqrt{s}]} \le \sup_{x\in \R} (\bm(x)+\calA(x)-x^2)+\e^{1/3}\mu b\sqrt{s}.$$
	Taking $\mu=\e^{-1/4}\sqrt{s}$, one can ensure the drift term above is negligible compared to $s$.  In fact, for this choice of $\mu$ and $a$, we have $\Pr(\EE^{\mu}(a)^c) \le \exp(-\frac23s^{3/2})$. To prove this, one relies on exponential type tail estimates for the argmax $Z(\pm a), Z^{\mu}(\pm a)$ that are developed in Section \ref{sec3} using parabolic decay estimates of directed landscape (Lemma \ref{lem:bbv}).

	Thus  \eqref{eq:arek} for $\nabla\h_{\e}^{[0,b\sqrt{s}]}$ now follows from upper tail asymptotics of Baik-Rains distribution which is given by \cite{ferrari2021upper}. An analogous argument provides the same upper tail estimate for $\nabla\h_{\e}^{[-b\sqrt{s},0]}$. This justifies item \ref{item2}. Item \ref{item1} appears as Lemma \ref{lem:maxbound} later. Its proof also involves the above Brownian replacement trick shown and tail estimates for the argmax. For brevity, we skip the details here and encourage the readers to read the details in Section \ref{sec3}.

	The above description of the main argument for proving \eqref{eq:arek} is of course quite reductive, and the full argument, presented in Section \ref{sec4.3}, relies on various technical estimates related to the location of the argmax that are discussed in Section \ref{sec3}. Apart from the temporal increment tail, these estimates on the argmax also allow us to extract other probabilistic information or the KPZ fixed point such as growth estimates (Proposition \ref{prop:gc}) and spatial modulus of continuity (Proposition \ref{prop:mc}).

	In fact, the precise statement of Proposition \ref{prop:stincrementtail} can handle absolute value of any small scaled increment in the vicinity of $1$. As a consequence, this leads to suitable temporal modulus of continuity type estimates in Corollary \ref{cor:pmod} and Proposition \ref{pp:weakc}. Combining the temporal modulus of continuity estimates along with sharp upper tail asymptotics of the increments leads to the upper bound for the short-time LIL.

	\medskip
	
	The lower bound for the short-time LIL on the other hand relies on showing a certain kind of independence structure in the temporal increments of the KPZ fixed point.  Loosely speaking we show that for two small increments  $\e_1\ll \e_2 \ll 1$, $\nabla \h_{\e_1}$ and $\nabla \h_{\e_2}$ are approximately independent. 
	Keeping $\nabla \h_{\e_2}$ as it is, we construct a \textit{proxy} for $\nabla \h_{\e_1}$ suitable for comparison with $\nabla \h_{\e_2}$. Towards this end we define
	\begin{align*}
		\nabla \h_{\e_1 \downarrow \e_2}:= \sup_{x\in \R} \left[\frac{\h_1(\e_1^{2/3}x)-\h_1(0)}{\e_1^{1/3}}+\frac{\cl(\e_1^{2/3}x,1+\e_2;0,1+\e_1)}{\e_1^{1/3}}\right].
	\end{align*}
	In plain words, we obtain $\nabla \h_{\e_1 \downarrow \e_2}$ from $\nabla \h_{\e_1}$ by na\"ively \textit{replacing the directed landscape} $\cl(\e_1^{2/3}x,1; 0,1+\e_1)$ by $\cl(\e_1^{2/3}x,1+\e_2; 0,1+\e_1)$ (and hence the name `Landscape Replacement').  Note that the proxy $\nabla \h_{\e_1 \downarrow \e_2}$ is same as $\nabla \h_{\e_1-\e_2}$ in distribution and furthermore the landscape part of the proxy is independent from that of $\nabla\h_{\e_2}$. As $\frac{\e_1}{\e_2}\to \infty$, this proxy turns out to be very good approximation of $\nabla \h_{\e_1}$ in the following sense.
	
	\begin{theorem}[Simpler version of Theorem \ref{thm:lr}]\label{slr} There exists a constant $\Con>0$ depending on the initial data  such that for $\e_2\ll \e_1\ll 1$  we have 
		\begin{align} \label{eq:simp}
			\Pr\left(|\nabla \h_{\e_1}-\nabla \h_{\e_1 \downarrow \e_2}|\ge 1\right) \le \Con\exp\left(-\tfrac1{\Con}\left(\tfrac{\e_1}{\e_2}\right)^{\frac3{16}}\right).
		\end{align} 
	\end{theorem}
	
	Due to the presence of $\frac{\e_1}{\e_2}$ term on the r.h.s.~of \eqref{eq:simp}, this suggests $\nabla \h_{\e_1}$ and $\nabla \h_{\e_2}$ are approximately independent when $\e_1$ and $\e_2$ are far apart on a \textit{multiplicative scale}, i.e., $\frac{\e_1}{\e_2} \gg 1$.  This suggests the decorrelation of the temporal increments happens at a slow rate and thus produces a law of iterated logarithm behavior.
	
	\smallskip
	
	Notice that the above theorem only ensures the landscape part of $\nabla\h_{\e_1}$ and $\nabla\h_{\e_2}$ appearing in the variational problem in \eqref{eq:arek1} are approximately independent. We now briefly explain how the rescaled spatial processes appearing in corresponding variational problems are also approximately independent. This part of the argument proceeds similarly to the proof of Proposition \ref{sims} explained above. Indeed, thanks to the Brownian replacement trick explained earlier, we can again transfer to Brownian motion with drifted initial data. The drift term can again be ignored due to the presence of diffusive scaling. Restricting the supremum to appropriate disjoint intervals and using independence of the increments for Brownian motion, one can obtain an independence structure within the temporal increments of the KPZ fixed point. {The details of the proof of short time landscape replacement and the extraction of lower bound from it are presented in Section \ref{sec5.2} and Section \ref{sec:slwld} respectively.}   
	
	\smallskip
	
	{For the long-time LIL, one of the key input is the sharp one point tail estimates for $t^{-1/3}\h_t(0)$, available from \cite{ferrari2021upper}. Combining this with the temporal modulus of continuity estimates (Proposition \ref{prop:stincrementtail}) discussed earlier, produces the upper bound for the long-time LIL result. For the lower bound, we rely on a similar \emph{landscape replacement} trick as discussed above. Given two temporal KPZ fixed point height functions, $\h_{t_1}(0)$ and $\h_{t_2}(0)$ with $t_2\gg t_1$, as before we construct a proxy for $\h_{t_2}(0)$ by suitably changing the directed landscape part of the variational problem in \eqref{def:kpzfp}. More precisely, we define the proxy to be
		$$\h^{t_2\downarrow t_1}:=\sup_{x\in \R} \left[\h_0(x)+\cl(x,t_1;0,t_2)\right].$$
		Just as in Theorem \ref{slr}, this proxy turns out to be good approximation $\h_{t_2}(0)$ in the following sense.
		\begin{theorem}[Simpler version of Theorem \ref{ind}] There exists a constant $\Con>0$ depending on the initial data  such that for $t_2\gg t_1\gg 1$  we have 
			\begin{align*}
				\Pr\left(t_2^{-1/3}|\h_{t_2}(0)- \h^{t_2 \downarrow t_1}|\ge 1\right) \le \Con\exp\big(-\tfrac1{\Con}\big(\tfrac{t_2}{t_1}\big)^{\frac1{3}}\big).
			\end{align*} 
		\end{theorem}
		Proof of the above result relies on precise estimate on the location of the argmax for the variational problem in \eqref{def:kpzfp}. This estimate is captured in Proposition \ref{lem:pr-tail1}. In the case of non-random initial data (i.e., $\h_0$ is non-random), the proxy $\h^{t_2\downarrow t_1}$ is independent of $\h_{t_1}(0)$. For Brownian initial data, the independence structure is not apriori clear since $\h^{t_2\downarrow t_1}$ and $\h_{t_1}(0)$ are coupled via the initial data $\h_0(\cdot)$. However, \emph{an almost independence structure} can be obtained by restricting the supremum on appropriate disjoint intervals. We refer the readers to Section \ref{sec5.1} and Section \ref{sec:lblil} for the technical details of the above conceptual argument.}
	
	\smallskip
	

	{As a passing comment, we mention that this landscape replacement idea previously appeared in a different guise for the KPZ equation with narrow wedge initial data in \cite{corwin2021kpz,dg21}. \red{For instance, \cite{dg21} (see Proposition~6.1) developed an analogue of landscape replacement trick for showing the law of iterated logarithms for the KPZ equation as time variable goes to $\infty$. Two main tools in \cite{dg21} are $(1)$ the convolution formula for the KPZ equation which is often regarded as the positive temperature analogue of the variational description of the KPZ fixed point and, $(2)$  the $\mathbf{H}$-Brownian Gibbs property of KPZ line ensemble \cite{CH16}. While the convolution formula of the KPZ equation is valid starting from any initial data, the $\mathbf{H}$-Brownian Gibbs property fails to hold when KPZ equation is started from an initial data which is not narrow wedge. This restricts \cite{dg21} to extend their results beyond the narrow wedge initial condition. In contrast to \cite{dg21}, the long time LILs  of Theorem~\ref{thm:longLILa} and~\ref{thm:longLILb} covers a large class of initial data by applying landscape replacement trick followed by using rich geometric structure of the directed landscape. On the other hand, the short time LIL of Theorem~\ref{thm:main} relies on the aforementioned Brownian replacement trick instead. This latter trick which reflects locally Brownian nature of the KPZ fixed point is very much different than the replacement tricks used in \cite{dg21}. In fact, the Brownian replacement trick allows us to handle increment of the spatial profile for a large class of initial data just based on one point tail estimates of the location of argmax in the variational formula of \eqref{def:kpzfp}, a property which is currently beyond reach for the KPZ equation or any other positive temperature models in the KPZ universality class. }
	

}
	
	\subsection{Literature review}	
	
	Our main results on the growth of tall temporal peaks and  local temporal growth of the KPZ fixed point draw inspirations from a series of efforts in understanding the geometry of the KPZ fixed point and the directed landscape. Being an universal scaling limit, the KPZ fixed point enjoys several rich properties and exhibits remarkable fractal geometry. For example, for all sub-parabolic initial data, $\h_t(x)$ is locally
	absolutely continuous with respect to Brownian motion for each $t>0$, and H\"older $1/3^{-}$ continuous in $t$ \cite{ch14,sv21}. Thus for rapidly decaying initial data, for each $t>0$ the $\h_t(\cdot)$ has a unique maximizer almost surely \cite{fqr13,ch14,pim14,sv21,chmm}. In this context, a recent line of inquiries includes the investigation of
	the fractal nature of exceptional times where $\h_t(\cdot)$ has multiple maximizers \cite{chmm,dau22}. 
	
	\smallskip

	Fractal geometry of the directed landscape and its geodesics has also been studied extensively in recent years (see \cite{g22} for a partial survey). In \cite{BGH19} the author initiated the study of the difference profile of the directed landscape $\calD(x,t):=\cl(1,0,x,t)-\cl(-1,0,x,t)$. Difference profiles exhibit rich fractal behavior \cite{BGH19,gz22}, has connection to disjointness of directed geodesics \cite{BGH20}, and are related to Brownian local time \cite{gh21}. Certain qualitative features of geodesics are also explored \cite{dsv20}.
	
	\smallskip
	
	Apart from the KPZ fixed point, long time law of iterated logarithms has also been rigorously shown in certain pre-limiting models in the KPZ universality class. For exponential last passage percolation, \cite{ledoux2018law,basu2019lower} identified correct scaling of the law of iterated logarithms for the $\limsup$ and $\liminf$ of the rescaled passage time. In the context of random matrices, \cite{paquette2017extremal} proved a law of fractional logarithm for the $\limsup$ of the sequence of eigenvalues in a GUE minor process.
	
	\smallskip
	
	For the KPZ equation, the precise long time LIL in the temporal direction was obtained in \cite{dg21}. Apart from the results related to LILs, \cite{dg21} also demonstrated certain fractal behavior of the level sets of the peaks for the KPZ equation. Understanding fractal nature of such level sets is important in the context of non-linear SDEs as it is intimately connected with the phenomenon of intermittency (see \cite{KKX17,KKX18, dg21, j21} are references there in).
	
	\smallskip
	
	Short time LIL for the temporal increment of the KPZ equation has been recently proved in \cite{das22}. The exponents for such LIL for the KPZ equation is marked different from our Theorem \ref{thm:main} as the local temporal structure for the KPZ equation is governed by a fractional Brownian motion of index $\frac14$ (see \cite{khoshnevisan2013weak,das22}).	
	
	\subsection*{Outline} The rest of the paper is organized as follows. In Section \ref{sec2} we collect several interesting properties and estimates for the directed landscape and the KPZ fixed point. In Section \ref{sec3} we determine quantitative estimates for the location of the supremum of the variational problem of the KPZ fixed point. These estimates appear in Proposition \ref{lem:pr-tail1}, Corollary \ref{lem:exitbound} and Proposition \ref{lem:maxbound}. In Section \ref{sec4}, we establish growth estimates (Proposition \ref{prop:gc}), spatial modulus of continuity (Proposition \ref{prop:mc}), and and temporal modulus of continuity (Corollary \ref{cor:pmod}) estimates for the KPZ fixed point. In Section \ref{sec5} we prove our long time and short time landscape replacement theorems that provide independent proxies for the KPZ fixed point (see the explanation in Section \ref{sec1.1}). Finally in Section \ref{sec6} we prove our main results: Theorem \ref{thm:longLILa}, Theorem \ref{thm:longLILb}, and Theorem \ref{thm:main}. Proof of two technical lemmas are deferred to Appendix \ref{app}.

	\subsection*{Acknowledgements} The authors thank Ivan Corwin for helpful comments on an earlier draft of the paper. The authors acknowledge support from NSF DMS-1928930 during their participation in the program ``Universality and Integrability in Random Matrix Theory and Interacting Particle Systems'' hosted by the Mathematical Sciences Research Institute in Berkeley, California in fall 2021.  SD’s research was also partially supported by Ivan Corwin's NSF grant 	DMS-1811143 and the Fernholz Foundation’s ``Summer Minerva Fellows'' program.
	
	\section{Basic framework and tools} \label{sec2}
	
	In this section, we collect several existing properties for the directed landscape and the KPZ fixed point. These two objects admit several useful symmetries and distributional relations to well known objects. In Section \ref{sec2.1} we collect all such facts that will come handy in our later analysis. In Section \ref{sec:dl-decay} we then recall probabilistic estimates for the one point tail of the KPZ fixed point and provide parabolic decay estimates and modulus of continuity estimates for the directed landscape. Finally in Section \ref{sec2.3} we collect one point tail information for the KPZ fixed point.
	
	\medskip
	
	Throughout this paper, we reserve the letters $\cl(\cdot,\cdot;\cdot,\cdot)$, $\calA(\cdot)$, and $\B(\cdot)$ to denote the directed landscape, stationary $\operatorname{Airy}_2$ process, and a two sided Brownian motion with diffusion coefficient $2$ respectively. We use $\Con = \Con(a, b, c, \ldots) > 0$ to denote a generic deterministic positive finite constant that may change from line to line, but is dependent on the designated variables $a, b, c, \ldots$.

	\subsection{Symmetries in the directed landscape and the KPZ fixed point} \label{sec2.1}
	
	The directed landscape as defined in 
	Definition \ref{def:dl} satisfies several symmetries. We state these results in terms of the stationary version of directed Landscape defined as
	\begin{align}\label{def:ck}
		\ck(x,s;y,t):=\cl(x,s;y,t)+\frac{(x-y)^2}{t-s}.
	\end{align}

	\begin{lemma}[Lemma 10.2 in \cite{DOV18}] \label{lem:sym} Fix any $z,r\in \R$ and $c,q>0$. We have the following equality in distributions as functions in $C(\R_{\uparrow}^4, \R)$.
		\begin{enumerate}[label=(\alph*), leftmargin=15pt]
			\item $\ck(x,s;y,t)\stackrel{d}{=} \ck(x,r+s;y,r+t)$.
			\item \label{2.1b} $\ck(x,s;y,t)\stackrel{d}{=} \ck(x+z,r;y+z,t)$.
			\item \label{2.1c} $\ck(x,s;y,t)\stackrel{d}{=} \ck(x+cs,s;y+ct,t)$.
			\item \label{2.1d} $\ck(x,s;y,t)\stackrel{d}{=} q\ck(q^{-2}x,q^{-3}s;q^{-2}y,q^{-3}t)$.
		\end{enumerate}
	\end{lemma}

	We reserve the notation $\ck(\cdot,\cdot;\cdot,\cdot)$ for the stationary directed landscape defined in \eqref{def:ck} for the rest of the article. The marginals of $\ck$ are given by the following lemma.
	
	\begin{lemma} \label{lem:2.2}
		\begin{enumerate}[label=(\alph*), leftmargin=15pt]
			\item For each fixed $y$, we have $\ck(\cdot,0;y,1)\stackrel{d}{=} \calA(\cdot-y)$, where $\calA(\cdot)$ is the stationary $\operatorname{Airy}_2$ process constructed in \cite{ps02}.
			\item For each fixed $x\in \R$ we have $\calA(x) \stackrel{d}{=} \tw$ where $\tw$ denotes the Tracy-Widom GUE distribution \cite{tw}.
		\end{enumerate}
	\end{lemma}
	\begin{proof}
		The first fact follows from the definition of the directed landscape (Definition \ref{def:dl}) and Proposition 14.1 from \cite{dv22}. The second fact is due to \cite{bary,gary}.
	\end{proof}
	
	Finally we discuss some useful identities for the KPZ fixed point started from narrow wedge or Brownian initial data.
	
	\begin{lemma}\label{lem:2.3} Consider the KPZ fixed point $\h$ started from an initial data $\h_0$. We have the following identities for the KPZ fixed point for special cases of $\h_0$.
		\begin{enumerate}[label=(\alph*), leftmargin=15pt]
			\item When $\h_0(x)=-\infty\cdot\ind_{x\neq 0}$, we have $\h_t(x)=\cl(0,0;x,t)$.
			\item \label{2.3b} When $\h_0(x)=\B(x)$, for each $t>0$ we have
			$t^{-1/3}\h_t(0) \stackrel{d}{=} \h_1(0) \stackrel{d}{=} \BR$
			where $\BR$ is a Baik Rains distributed random variable \cite{br} defined as $\BR:=\sup_{x\in \R} [\B(x)+\calA(x)-x^2].$
			\item \label{2.3c} Fix any $\mu\in \R$. When $\h_0(x)=\B(x)+\mu x$, as processes in $x$ we have $\h_s(x)-\h_s(0) \stackrel{d}{=} \B(x)+\mu x$.
		\end{enumerate}
	\end{lemma}
	\begin{proof}
		The first one is obvious from the definition of the KPZ fixed point (see \eqref{def:kpzfp}). The second fact follows from Lemma \ref{lem:sym} \ref{2.1d} and scale invariance of $\B$ (i.e., $r^{-1}\B(r^2x)\stackrel{d}{=} \B(x)$). The third fact is well known (see Theorem 4.5 in \cite{MQR16}).
	\end{proof}
	\subsection{Probabilistic estimates for the directed landscape}  \label{sec:dl-decay} In this section we prove a couple of probabilistic estimates for the directed landscape: Lemma \ref{lem:scont} and Lemma \ref{lem:bbv}. The first one collects several tail estimates for the spatial process of the directed landscape, whereas the second one establishes parabolic decay estimates for the directed landscape.

	\begin{lemma}[Spatial tail estimate]\label{lem:scont}  There exists a constant $\Con>0$ such that for any $k_1,k_2\in \R, s>0$, and $R\ge 1$ we have
		\begin{align}\label{eq:ProbBd1}
			& \Pr\Big(\sup_{y,z\in \mathfrak{I}_{k_1,k_2}} |\ck(y,0;z,1)-\ck(k_1,0;k_2,1)| \ge s\Big)  \le \se{2}, \\ 
			& \Pr\Big(\sup_{y,z\in \mathfrak{I}_{k_1,k_2}} |\ck(y,0;z,1)| \ge s\Big)  \le \se{3/2} \label{eq:ProbBd2}
			\\ 
			& \Pr\Big(\sup_{|y|\le R,|x|\le 1, |z|\le 1} \frac{|\ck(y,0;x,1)-\ck(y,0;z,1)|}{\sqrt{|x-z|\log \frac{4}{|x-z|}}} \ge s\Big)  \le R \cdot \se{3/2}, \label{eq:ProbBd3}
		\end{align}
		where $\mathfrak{I}_{k_1,k_2} = [k_1, k_1+1]\times [k_2,k_2+1]$. 
	\end{lemma}
	\begin{proof}
		By Lemma \ref{lem:sym} \ref{2.1b} and \ref{2.1c} it suffices to prove the lemma for $k_1=k_2=0$.     
		By Lemma 10.4 of \cite{DOV18}, there exists $\Con>0$ such that 
		\begin{equation}\label{eq:difest}
			\begin{aligned}
				& \Pr\left(|\ck(y_1,0;0,1)-\ck(y_2,0;0,1)| \ge s\sqrt{|y_1-y_2|}\right) \le \se{2}, \\ & \Pr\left(|\ck(0,0;z_1,1)-\ck(0,0;z_2,1)| \ge s\sqrt{|z_1-z_2|}\right) \le \Con\exp(-\tfrac1{\Con}s^2),
			\end{aligned}
		\end{equation}
		for all $y_1,y_2,z_1,z_2\in \R$. Applying Lemma 3.3 of \cite{dv21}, we have
		\begin{align*}
			\Pr\Big(\sup_{|y|,|z|\le 1} \frac{|\ck(y,0;z,1)-\ck(0,0;0,1)|}{|y|^{\frac12}\log^{\frac12}\big(\frac{2}{|y|}\big)+|z|^{\frac12}\log^{\frac12}\big(\frac{2}{|z|}\big)}\ge s\Big)\le \Con\exp(-\tfrac1{\Con}s^2).
		\end{align*}
		From the above inequality, \eqref{eq:ProbBd1} follows by noting that there exists a absolute constant $\Con>0$ such that 
		\begin{align*}
			\sup_{y,z\in \mathfrak{I}_{0,0}} |\ck(y,0;z,1)-\ck(0,0;0,1)|\leq 	\Con\cdot\sup_{|y|,|z|\le 1} \frac{|\ck(y,0;z,1)-\ck(0,0;0,1)|}{|y|^{\frac12}\log^{\frac12}\big(\frac{2}{|y|}\big)+|z|^{\frac12}\log^{\frac12}\big(\frac{2}{|z|}\big)}. 
		\end{align*}
		This completes the proof of \eqref{eq:ProbBd1}. For \eqref{eq:ProbBd2}, we utilize the fact that $\ck(0,0;0,1)\stackrel{d}{=} \tw$ from Lemma \ref{lem:2.2}, and $\Pr(|\tw|\ge s)\le \se{3/2}$ from \cite{ramirez2011beta}. Hence in view of \eqref{eq:ProbBd1}, by union bound we have
		\begin{align*}
			\Pr\Big(\sup_{y,z\in \mathfrak{I}_{0,0}} |\ck(y,0;z,1)| \ge s\Big) &  \le \Pr\Big(\sup_{y,z\in \mathfrak{I}_{0,0}} |\ck(y,0;z,1)-\ck(0,0;0,1)| \ge \tfrac{s}{2}\Big)+ \Pr(|\ck(0,0;0,1)|\ge \tfrac{s}2) \\ & \le  \se{3/2}.
		\end{align*}
		This establishes \eqref{eq:ProbBd2}. Finally \eqref{eq:ProbBd3} follows from another application of Lemma 3.3 in \cite{dv21} along with \eqref{eq:difest}.
	\end{proof}
	\begin{lemma}[Parabolic Decay of $\cl$]\label{lem:bbv}  	Fix $\delta>0$. There exist a constant $\Con=\Con(\delta) >0$ such that for all $r>0, v>0,$ and $s>0$ we have
		\begin{align*}
			\Pr\Big(\sup_{y\in \R, |z|\le r} \big[ \cl(y,0;z,v)+\tfrac{(y-z)^2}{v(1+\delta)}\big] \ge s\Big) \le \Con (rv^{-\frac23}+1)\exp(-\tfrac{1}{\Con}s^{3/2}v^{-\frac12}).
		\end{align*}
	\end{lemma}
	\begin{proof} Recall that the KPZ fixed point satisfies the following scale invariance property: for any fixed $v>0$,
		\begin{align*}
			\mathcal{L}(y,0;z,v)+ \frac{(y-z)^2}{v(1+\delta)} \stackrel{d}{=} v^{\frac{1}{3}}\Big[\mathcal{L}(v^{-\frac{2}{3}}y,0;v^{-\frac{2}{3}}z,1)+ \frac{\big(v^{-\frac{2}{3}}(y-z)\big)^2}{(1+\delta)}\Big].
		\end{align*}
		By this scale invariance property, we have
		\begin{align*}
			\Pr\Big(\sup_{y\in \R, |z|\le r} \big[ \cl(y,0;z,v)+\tfrac{(y-z)^2}{v(1+\delta)}\big] \ge s\Big) = \Pr\Big(\sup_{y\in \R, |z|\le rv^{-\frac23}} \big[ \cl(y,0;z,1)+\tfrac{(y-z)^2}{1+\delta}\big] \ge sv^{-\frac13}\Big).
		\end{align*}
		Thus it suffices to prove Lemma \ref{lem:bbv} for $v=1$. To this end, for any $k_1,k_2\in \Z$ we define,
		\begin{align*}
			\mathsf{A}_{k_1,k_2} &:= \Big\{\sup_{y\in [k_1,k_1+1], z\in [k_2,k_2+1]} \big[ \ck(y,0;z,1)  -\tfrac{\delta(y-z)^2}{1+\delta}\big] \ge s\Big\},\\
			\mathsf{B}_{k_1,k_2} &:= \Big\{ \sup_{\substack{y\in [k_1,k_1+1]\\ z\in [k_2,k_2+1]}} |\ck(y,0;z,1)| \ge s+\tfrac{\delta m_{k_1,k_2}^2}{1+\delta}\Big\}, \quad \mbox{where }  m_{k_1,k_2}^2 := \min_{\substack{y\in [k_1,k_1+1] \\ z\in [k_2,k_2+1]}}(y-z)^2.
		\end{align*}
		Observe that by union bound we have
		\begin{align}\label{eq:BreakIt}
			\Pr\Big(\sup_{y\in \R, |z|\le r} \big[ \cl(y,0;z,1)+\tfrac{(y-z)^2}{1+\delta}\big] \ge s\Big) \leq \sum_{k_1\in \mathbb{Z}}\sum_{k_2 = - \lceil r \rceil}^{\lceil r \rceil}\mathbb{P}(\mathsf{A}_{k_1,k_2}) \leq \sum_{k_1\in \mathbb{Z}}\sum_{k_2 = - \lceil r \rceil}^{\lceil r \rceil}\mathbb{P}(\mathsf{B}_{k_1,k_2}).  
		\end{align}	 	
		From Lemma \ref{lem:scont} (\eqref{eq:ProbBd2} precisely) we  see that 
		\begin{align*}
			\mathbb{P}(\mathsf{B}_{k_1,k_2})\leq \Con\exp(-\tfrac1{\Con}(s+\tfrac{\delta}{1+\delta}m_{k_1,k_2}^2)^{3/2}). 
		\end{align*}
		We substitute this into the right hand side of \eqref{eq:BreakIt}. Summing over $k_1\in \Z$ and then $k_2 \in [-r-1,r+1]\cap \Z$, we get the desired result.
	\end{proof}
	
	\subsection{Tail estimates for the KPZ fixed point} \label{sec2.3} In this section we collect sharp one point tail estimates for the KPZ fixed point that is indispensable for proving our main results. 
	
	\smallskip
	
	To systematically study different initial data considered in our main theorems, we now introduce two classes for initial data, namely \st \ and \lt, suitable for short-time LIL and long-time LIL respectively. 
	
	\begin{definition}[\st] \label{def:st} Fix any $\stp =(\af,\dt,\bet)\in  \R\times \R_{>0}^2$. ST-class($\mathfrak{p}$) is a class of initial data for KPZ fixed point which is the union of the following sub-classes of initial data.  
		\begin{itemize}
			\item $\h_0(x)=-\infty\cdot\ind_{x\neq 0},$
			\item $\h_0$ is deterministic functional initial data with $\h_0(x) \le \af+(1-\dt)x^2$ for some $\af\in\R$ and $\dt>0$ and $|\h_0(0)|\le \bet$,
			\item $\h_0(x)=\B(x),$ a two sided Brownian Motion with diffusion coefficient $2$.
		\end{itemize}
	\end{definition}
	\begin{definition}[\lt] \label{def:lt} Fix any $\ltp =(\csq,\bet)\in  \R_{>0}^2$. LT-class($\mathfrak{q}$) is a class of initial data for KPZ fixed point which consists of the following sub-classes.
		\begin{itemize}
			\item $\h_0(x)=-\infty\cdot\ind_{x\neq 0},$
			\item $\h_0$ is deterministic with $\h_0(x) \le \csq\sqrt{1+|x|}$ for some $\csq>0$ and $|\h_0(0)|\le \bet$.
			\item $\h_0(x)=\bm(x)$, a two sided Brownian Motion with diffusion coefficient $2$.
		\end{itemize}
	\end{definition}
	The  KPZ fixed point is well defined for the above two classes of initial data \cite{sv21}. We now state the sharp one point tail asymptotics for the KPZ fixed point for different types of initial data within \st.
	\begin{lemma}\label{lem:onepttail} Fix any $\epsilon>0$. 
		\begin{enumerate}[label=(\alph*), leftmargin=15pt]
			\item Fix any $\stp=(\af,\dt,\bet) \in \R\times \R_{>0}^2$. Let $\h_t$ be the KPZ fixed point starting from an initial data $\h_0$ from the \st. There exists a constant $\Con(\bet)>0$ such that for all $s>0$ and $t\ge 1$ we have
			\begin{align*}
				\Pr(t^{-1/3}\h_t(0)\le -s) \le \Con\exp\left(-\tfrac1\Con s^{3}\right).
			\end{align*}
			
			\item Let $\h_t$ be the KPZ fixed point started from a functional initial data $\h_0$ from the \st. There exist constants $s_0(\e,\h_0)>0$, $\rho(\h_0)>1$ such that for all $s\ge s_0$ and $t\in [1,\rho]$ we have
			\begin{align}\label{eq:tbdg}
				\exp\left(-(\tfrac43+\epsilon)s^{3/2}\right) \le \Pr(t^{-1/3}\h_t(0)\ge s) \le \exp\left(-(\tfrac43-\epsilon)s^{3/2}\right).
			\end{align}
			\item \label{2.6c} Let $\h_t$ be the KPZ fixed point starting from $\h_0(x)=\B(x)$. There exists a constant $s_0(\e)>0$ such that for all $s\ge s_0$ and $t>0$ we have
			\begin{align}\label{eq:tbdg0}
				\exp\left(-(\tfrac23+\epsilon)s^{3/2}\right) \le \Pr(t^{-1/3}\h_t(0)\ge s)=\Pr(\BR \ge s) \le \exp\left(-(\tfrac23-\epsilon)s^{3/2}\right).
			\end{align}
			where $\BR$ is a Baik-Rains distributed random variable.
			\item \label{2.6d} Let $\B$ be a Brownian motion with variance $2$. There exists a constant $s_0(\e)>0$ such that for all $s\ge s_0$ and all interval $[a,b]$ containing $\sqrt{s}/2$ we have
			\begin{align*}
				\Pr( \sup_{z\in [a,b]} [\B(z)+\cl(z,0;0,1)] \ge s) \ge \exp\left(-(\tfrac23+\epsilon)s^{3/2}\right).
			\end{align*}
		\end{enumerate}
	\end{lemma}
	\begin{proof}
		Note that $$t^{-1/3}\h_t(0) \ge -\bet+t^{-1/3}\cl(0,0;0,t) \stackrel{d}{=} -\bet+\tw.$$
		Part (a) now follows by utilizing the lower tail estimates for the $\tw$ distribution \cite{ramirez2011beta}. The equality in Part (c) follows from Lemma \ref{lem:2.3} \ref{2.3b}. The estimates in part (b) and (c) of the above lemma are weak versions of Theorems 1.1 and 1.3 in \cite{ferrari2021upper} respectively. For part (d) using Brownian tail and Tracy-Widom tail estimates \cite{ramirez2011beta}, observe that for all $s$ large enough 
		\begin{align*}
			\Pr( \sup_{z\in [a,b]} \B(z)+\cl(z,0;0,1) \ge s) & \ge \Pr( \B({\sqrt{s}/2})+\cl(\sqrt{s}/{2},0;0,1) \ge s)  \ge \Pr(\B({\sqrt{s}/2})\ge s)\Pr(\tw \ge \tfrac{s}4).  
		\end{align*}
	Note that the right side of the last inequality in the above display is bounded below by $e^{-\frac{1+\e}{2}s^{3/2}}e^{-\frac{1+3\e}6s^{3/2}}=e^{-(\frac23+\epsilon)s^{3/2}}$. This concludes the proof.
	\end{proof}

	\section{Properties of the Argmax}\label{sec3}

	Recall that given an initial data $\h_0$, the KPZ fixed point is given by the following variational problem:
	\begin{align}\label{def:kpzfp2}
		\h_t(x)=\sup_{z\in \R} (\h_0(z)+\cl(z,0;x,t)).
	\end{align}
	The goal of this section is to provide quantitative estimates for the location of the above supremum. These estimates will allow us to restrict the above supremum over appropriate compact intervals, which in turn will be utilized in Section \ref{sec4} to extract several probabilistic properties of the KPZ fixed point.  
	
	Given an arbitrary initial data $\h_0$, we denote $Z_t(x;\h_0)$ to be rightmost maximizer of the r.h.s.~of \eqref{def:kpzfp2},
	\begin{align}\label{def:ztx}
		Z_t(x;\h_0):=\sup \underset{z\in \R}{\operatorname{argmax}} (\h_0(z)+\cl(z,0;x,t)).
	\end{align}
	We will be also interested in a special case of the above definition, when initial data is given by a Brownian motion with a given drift. We use a separate notation for this case:
	\begin{align}
		\label{def:zux}
		Z_t^{\mu}(x):=\sup \underset{z\in \R}{\operatorname{argmax}} (\B(z)+\mu \cdot z+\cl(z,0;x,t)).
	\end{align}
	
	The following lemma discusses some useful distributional identities for $Z_t^{\mu}(a)$ and a comparison principle which allows us to compare increments of the KPZ fixed point with arbitrary initial data with that of Brownian motion with drift as initial data.
	
	\begin{lemma}\label{lem:argmaxprop} Consider the KPZ fixed point $\h_t$ started from an initial data $\h_0$ in \st. For each $\mu\in \R$, consider the KPZ fixed point $\h_t^{\mu}$ started from initial data $\h_0(x)=\B(x)+\mu x$. Recall the corresponding maximizers from \eqref{def:ztx} and \eqref{def:zux}. We have the following.
		\begin{enumerate}[label=(\alph*), leftmargin=15pt]
			\item \label{3.1a} For every $\mu,x \in \R$ and $t>0$ we have $Z_t^{0}(0)\stackrel{d}{=}t^{2/3}Z_1^0(0)$ and
			\begin{align}\label{eq:nt}
				Z_t^{\mu}(x) \stackrel{d}{=} Z_t^{0}(x+\tfrac12\mu t) \stackrel{d}{=} Z_t^{0}(0)+x+\tfrac12\mu t.
			\end{align}
			\item \label{lem:compare} Fix $a,\mu,t>0$ and consider the event
			\begin{align}\label{eq:etma}
				\EE_t^{\mu}(a):=\left\{Z_t(a;\h_0) \le Z_t^{\mu}(-a), Z_t(-a;\h_0) \ge Z_t^{-\mu}(a)\right\}.
			\end{align}
			On the event $\EE_t^{\mu}(a)$ we have for all $x,y \in [-a,a]$ with $x\le y$ we have
			\begin{align*}
				\h_t^{-\mu}(y)-\h_t^{-\mu}(x) \le \h_t(y)-\h_t(x) \le \h_t^{\mu}(y)-\h_t^{\mu}(x).
			\end{align*}
		\end{enumerate}

	\end{lemma}
	\begin{proof} For part (a), observe that as processes in $z$,
		\begin{align*}
			\B(z)+\cl(z,0;0,t) & \stackrel{d}{=} t^{1/3}\left[\B(t^{-2/3}z)+\cl(t^{-2/3}z,0;0,1)\right]. 
		\end{align*}
		Considering the rightmost maximizers for each of the above functions shows $Z_t^{0}(0)\stackrel{d}{=}t^{2/3}Z_1^0(0)$. For \eqref{eq:nt} note that by Lemma \ref{lem:sym} \ref{2.1c}
		\begin{align*}
			\B(z)+\mu\cdot z+\cl(z,0;x,t) &  =\B(z)+\mu\cdot z+\ck(z,0;x,t)-(z-x)^2/t \\ & \stackrel{d}{=} \B(z)+\mu\cdot z+\ck(z,0;x+\tfrac12\mu t,t)-(z-x)^2/t \\ & = \B(z)+\mu\cdot z+\cl(z,0;x+\tfrac12\mu t,t)-(z-x)^2/t+(z-x-\tfrac12\mu t)^2/t \\ & = \B(z)+\mu\cdot x+\cl(z,0;x+\tfrac12\mu t,t)+\tfrac14\mu^2 t.
		\end{align*}
		Considering the rightmost maximizers for the above functions yields the first part of \eqref{eq:nt}. The second distributional equality follows by considering the rightmost maximizers for the following two functions
		\begin{align*}
			\B(z)+\cl(z,0;x+\tfrac12\mu t,t) \stackrel{d}{=} \B(z-x-\tfrac12\mu t)-\B(-x-\tfrac12\mu t)+\cl(z-x-\tfrac12\mu t,0;0,t). 
		\end{align*}
		Part (b) is derived in Equation (4.9) in \cite{pim}.
	\end{proof}
	
	We now turn towards the task of estimating the location of the argmax. We first prove a technical lemma related to it which argues when the supremum in \eqref{def:kpzfp2} is restricted to $|z|\ge r$, with high probability the supremum value is less than $-(\operatorname{const}) \cdot r^2$.

	\begin{proposition}
		\label{lem:pr-tail1} Fix any $\stp \in \R\times \R_{>0}^2$. Let $\h_0$ be an initial data of the KPZ fixed point from the \st. There exist constants $\rho=\rho(\stp)\in (1,2]$ and $\Con=\Con(\stp)>0$ such that for all $t\in [1,\rho]$ and $|x|\le \Con^{-1}$,
		\begin{align}\label{eq:ConvTail3}
			\Pr\Big(\sup_{|z|\ge r} \big[\h_0(z)+\cl(z,0;xr,t)\big] \ge -\tfrac1\Con r^2\Big) \le \Con\exp\left(-\tfrac1\Con r^3 \right).
		\end{align}
	\end{proposition}
	\begin{proof} Note that when $\h_0$ is the narrow wedge at $0$, then, \eqref{eq:ConvTail3} is trivial as 
		\begin{align*}
			\sup_{z\in \mathcal{C}}\big[\h_0(z)+\cl(z,0;xr,t)\big] =-\infty
		\end{align*}
		for any subset $\mathcal{C} \subset \mathbb{R} $ not containing $0$. So, it suffices to prove the result for functional and Brownian initial data from the \st. 
		
		Fix any $\stp \in \R\times \R_{>0}^2$. We first assume $\h_0$ is a functional data in \st. Set 
		\begin{align*}
			\rho:=\begin{cases}
				\dfrac{1}{1-\frac14\dt} & \mbox{ if } \dt\le 2\\
				2 & \mbox{ otherwise}.
			\end{cases}
		\end{align*}
		
		Set $\Con_1>0$ large enough so that $x^2+2|x| \le \tfrac{\min(\dt,1)}{16}$ for all $|x|\le \Con_1^{-1}$. Without loss of generality we may assume $r\ge 4\af/\dt$. Otherwise the inequality follows trivially by choosing $\Con$ appropriately large in terms of $\stp$. Note that for all $|z|\ge r\ge 4\af/\dt$ and $t\in [1,\rho]$ we have  
		\begin{align}\label{eq:a1}
			\h_0(z) \le \af+(1-\dt)z^2 & \le -\tfrac14\dt \cdot r^2+(1-\tfrac12\dt)z^2 \\ & \le -\tfrac14\dt \cdot r^2+\frac{z^2}{t(1+\frac14\dt)}. \nonumber
		\end{align}
		For $t\ge 1$, by Lemma \ref{lem:sym} \ref{2.1c} for all $|x|\le \Con_1^{-1}$ we have 
		\begin{align*}
			\cl(z,0;xr,t)\stackrel{d}{=}\cl(z,0,0,t)+[z^2-(z-xr)^2]/t \le \cl(z,0;0,t) \le \tfrac1{16t}\min(\dt,1) z^2+\tfrac1{16}\dt r^2.
		\end{align*}
		Adding the preceding two displays yield
		\begin{align*}
			\Pr\Big(\sup_{|z|\ge r} \left(\h_0(z)+\cl(z,0;xr,t)\right) \ge -\tfrac1{16}\dt r^2 \Big) & \le \Pr\Big(\sup_{|z|\ge r} \left(\cl(z,0;0,t)+\tfrac{z^2}{t}\left[\tfrac{1}{1+\frac14\dt}+\tfrac{\min(\dt,1)}{16}\right]\right) \ge \tfrac1{8}\dt r^2 \Big) \\ & \le \Con_2\exp\big(-\tfrac1{\Con_2} r^3\big),
		\end{align*}
		where the last estimate follows from the parabolic decay estimate in Lemma \ref{lem:bbv}. Here we use the fact $\tfrac{1}{1+\frac14\dt}+\tfrac{\min(\dt,1)}{16}<1$. Choosing $\Con=\max\{\Con_1,\Con_2,\frac{16}{\dt}\}$ we get the desired result for functional initial data. 
		
		\medskip
		
		Now we suppose $\h_0(x)=\B(x)$ is a Brownian initial data. By Lemma 2.12 in \cite{CH16} we know
		\begin{align*}
			\Pr\left(\B(z) \ge \tfrac1{8}r^2+\tfrac18 z^2 \mbox{ for some } z\in \R \right) \le \Con\exp\left(-\tfrac1\Con r^3\right).
		\end{align*}
		This implies with probability at least $1-\Con\exp\left(-\tfrac1\Con r^3\right)$ for all $|z|\ge r$ we have
		\begin{align*}
			\B(z) \le \tfrac12z^2-\tfrac18r^2.
		\end{align*}
		The above fact is the analogue of \eqref{eq:a1} for Brownian initial data (with $\dt=1$). Following the same arguments as for functional data, \eqref{eq:ConvTail3} can now be established for Brownian data. This completes the proof.
	\end{proof}
	
	\begin{corollary}\label{lem:exitbound}
		Fix any $\stp \in \R\times \R_{>0}^2$. Consider the KPZ fixed point $\h$ started from an initial data $\h_0$ from the \st.  There exist constants  $\rho(\stp)\in (1,2]$ and $\Con(\stp)>0$ such that for all $r>0$, $t\in [1,\rho]$, and $|x| \le \Con^{-1}$  we have
		\begin{equation}\label{eq:exitb}
			\P\big(|Z_t(xr;\h_0)| \geq r\big) \leq  \Con \exp\left(-\tfrac1\Con r^{3}\right).
		\end{equation}
	\end{corollary}
	\begin{proof} Consider $\rho(\stp)\in (1,2]$ and $\Con(\stp) >0$ from Proposition \ref{lem:pr-tail1}. Without loss of generality assume $\Con\ge 2$.
		\begin{align*}
			\h_0(0)+\cl(0,0;xr,t) \stackrel{d}{=} \h_0(0)+t^{1/3}\tw-\tfrac{x^2r^2}{t}.
		\end{align*}
		As $t\ge 1$, we have $\frac{x^2r^2}{t} \le  \frac{1}{2\Con}r^2$. Thus as $|\h_0(0)|\le \bet$ by lower tail estimates for $\tw$ from \cite{ramirez2011beta} we get that
		\begin{align*}
			\Pr\left(\h_0(0)+\cl(0,0;xr,t) \le -\tfrac1{\Con}r^2\right) \le   \Pr\Big(t^{1/3}\tw \le -\tfrac1{2\Con}r^2+\bet\Big) \le \til\Con\exp\big(-\tfrac1{\til\Con} r^6\big).
		\end{align*}
		for some different constant $\til{\Con}(\stp)>0$. The above estimate shows that with high probability $\h_0(0)+\cl(0,0;xr,t) >-\frac{1}{\Con}r^2$. However by Lemma \ref{lem:pr-tail1} we know the supremum in \eqref{def:kpzfp2} when restricted to $|z|\ge r$ is less than $-\frac1{\Con}r^2$ with high probability. Thus setting $\Con_0:=2\cdot \max\{\Con,\til{\Con}\}$, in view of \eqref{eq:ConvTail3}, we thus arrive at \eqref{eq:exitb}. This concludes the proof.
	\end{proof}
	
	\begin{remark}\label{rm:anymax} Note that Corollary \ref{lem:exitbound} is stated for the rightmost maximizer. However, from Section 3.1 in \cite{pim} it is known that the location of the maxima in \eqref{def:kpzfp2} is almost surely unique. Thus \eqref{eq:exitb} can be restated as any maximizer exiting the interval $[-r,r]$ has probability at most $\Con \exp\left(-\frac1\Con r^{3}\right)$.
	\end{remark}
	
	Apart from the maximizer of the original variational problem in \eqref{def:kpzfp2}, we will also need an estimate for the maximizer for the variational problem of the increment: $$\h_t(0)-\h_s(0)=\sup_{z\in \R}\{ \h_s(z)-\h_s(0)+\cl(z,0;0,t-s)\}.$$ Note that the above increment upon scaling can be written as
	\begin{align*}
		\frac{\h_t(0)-\h_s(0)}{(t-s)^{\frac13}} =\sup_{z\in \R}\left[\frac{\h_s((t-s)^{\frac23}x)-\h_s(0)}{(t-s)^{\frac13}}+\calA(x)-x^2\right]
	\end{align*}
	where $\calA$ is a stationary $\operatorname{Airy}_2$ process independent of $\h$. We have the following quantitative estimate for the location of the maximizer for the above variational problem.
	\begin{lemma}\label{lem:maxbound}
		Fix any $\stp\in \R\times\R_{>0}^2$. Let $\h_0$ be either narrow wedge or functional initial data from \st. Consider the KPZ fixed point $\h$ started from initial data $\h_0$. Let $\calA$ be a stationary $\operatorname{Airy}_2$ process independent of $\h$. There exists $\rho=\rho(\stp)>1$ and $\Con=\Con(\stp)>0$ such that for all $1 \leq  s < t \leq \rho$ and  $r > 0$ we have 
		$$\P\Big(\sup_{|x|\ge r} \Big(\frac{\mathfrak{h}_s ((t-s)^{\frac{2}{3}} x) - \mathfrak{h}_s (0)}{(t-s)^{\frac{1}{3}}} + \calA(x)-x^2\Big) > \calA(0)\Big) \leq \Con \exp\big(-\tfrac{1}{\Con} r^3\big).$$
	\end{lemma}
	\begin{proof}
		When $\h_0$ is narrow wedge, as processes in $x$ we have
		\begin{align*}
			\frac{\mathfrak{h}_s ((t-s)^{\frac{2}{3}} x) - \mathfrak{h}_s (0)}{(t-s)^{\frac{1}{3}}} \stackrel{d}{=} \gamma\til\calA(x/\gamma^2)-\gamma\til\calA(0)-\gamma^{-3}x^2
		\end{align*}
		where $\til\calA(\cdot)$ is another stationary $\operatorname{Airy}_2$ process independent of $\calA$ and $\gamma:=s^{1/3}(t-s)^{-1/3}$. The result then follows from \cite[Lemma 9.5]{DOV18}. So, let us assume $\h_0$ is a functional initial data from \st. Without loss of generality let us assume $r\ge 2$. For each $k\in \R$, let us define the event
		\begin{align*}
			\Y_k:= \Big\{\sup_{x \in [k, k+1]} \Big(\frac{\mathfrak{h}_s ((t-s)^{\frac{2}{3}} x) - \mathfrak{h}_s (0)}{(t-s)^{\frac{1}{3}}} + \calA(x)-x^2\Big) \geq -\tfrac12k^2 \Big\}.
		\end{align*}
		We claim that there exists a constant $\Con(\stp)>0$ such that for all $k\in \Z$ with $|k|\ge 1$ we have
		\begin{align}\label{eq:claim1}
			\Pr(\Y_k) \le \Con\exp\left(-\tfrac1\Con k^3\right).
		\end{align}
		Assuming \eqref{eq:claim1}, we observe that by union bound we have
		\begin{align*}
			& \Pr\Big(\sup_{|x|\ge r} \Big(\frac{\mathfrak{h}_s ((t-s)^{\frac{2}{3}} x) - \mathfrak{h}_s (0)}{(t-s)^{\frac{1}{3}}} + \calA(x)-x^2\Big) > \calA(0)\Big) \\ & \le \sum_{k\in \Z, |k|\ge r-1} \left[\Pr(\Y_k)+\Pr(\calA(0) \le -\tfrac12k^2)\right] \\ & \le \sum_{k\in \Z, |k|\ge r-1} \left[\Con\exp\left(-\tfrac1\Con k^3\right)+\Con\exp\left(-\tfrac1\Con k^6\right)\right]  \le \Con\exp\left(-\tfrac1\Con r^3\right).
		\end{align*}
		Hence it suffices to prove \eqref{eq:claim1}. Recall the event $\EE_t^{\mu}(a)$ from \eqref{eq:etma}. Set $\mu:=\frac{k^2}{4(k+1)}$ and $a=(k+1)(t-s)^{2/3}$. By Lemma \ref{lem:argmaxprop} \ref{lem:compare},  under $\EE_t^{\mu}(a)$ we have for $x \in [0, k+1]$,
		\begin{equation*}
			\frac{\mathfrak{h}_s ((t-s)^{\frac{2}{3}} x) - \mathfrak{h}_s (0)}{(t-s)^{\frac13}} \leq \frac{\h_s^{\mu} ((t-s)^{\frac{2}{3}} x) - \h_s^{\mu}(0)}{(t-s)^{\frac13}} \stackrel{d}{=} \B(x)+\mu(t-s)^{1/3}x,  
		\end{equation*}
		where the last equality in distribution is interpreted as processes in $x$ and follows from Lemma \ref{lem:2.3} \ref{2.3c}.
		Hence, 
		\begin{align*}
			\Pr(\Y_k) & \le \Pr\big(\Y_k \cap \EE_s^{\mu}(a)\big)+\Pr\big(\EE_s^{\mu}(a)^c\big) \\ & \le \Pr\Big(\sup_{x\in [k,k+1]} \big(\B(x)+\calA(x)-x^2\big) \ge - \tfrac14k^2\Big)+\Pr\left(\EE_s^{\mu}(a)^c\right) \\ & \le \Pr\Big(\sup_{x\in [k,k+1]} \B(x)\ge  \tfrac14k^2\Big)+\Pr\Big(\sup_{x\in [k,k+1]} \calA(x) \ge \tfrac12k^2\Big)+\Pr\big(\EE_s^{\mu}(a)^c\big).
		\end{align*}
		The first term above is at most $\Con\exp(-\frac1\Con k^3)$ by Brownian tail estimates. By \eqref{eq:ProbBd2}, the second term above is also at most $\Con\exp(-\frac1\Con k^3)$. Thus to show \eqref{eq:claim1} it suffices to prove
		\begin{align}
			\label{eq:claim2}
			\Pr\left(\EE_s^{\mu}(a)^c\right) \le \Con\exp\left(-\tfrac1\Con k^3\right).
		\end{align}
		Towards this end recall the definition of $\EE_s^{\mu}(a)$ from \eqref{eq:etma}. By union bound we have
		\begin{equation}
			\label{e.l.maxbound4}
			\begin{aligned}
				\mathbb{P}(\EE_s^{\mu}(a)^c) & \le  \mathbb{P}\Big(Z_s^{\mu}(-a) \leq Z_s (a;\h_0)\Big)+\Pr\left(Z_s^{-\mu}(a) \ge Z_s(-a;\h_0)\right)\\
				&\leq \mathbb{P}\Big(Z^{0}_s(0) -a+\tfrac12\mu s \leq \tfrac14\mu\Big) +  \mathbb{P}\Big(Z_s (a; \h_0)\geq \tfrac{1}{4}\mu\Big) \\
				& \hspace{2cm}+\mathbb{P}\Big(Z^{0}_s(0) +a-\tfrac12\mu s \geq -\tfrac14\mu\Big) +  \mathbb{P}\Big(Z_s (-a; \h_0)\leq -\tfrac{1}{4}\mu\Big),
			\end{aligned}
		\end{equation}
		where the last line follows from union bound and distributional identities from Lemma \ref{lem:argmaxprop} \ref{3.1a}.
		Let us write $x=a/\mu$, so that 
		Note that $|x| \le 16(t-s)^{2/3}$. We first pick $\rho(\stp)>1$, $\Con(\stp)>0$ from Corollary \ref{lem:exitbound}. Then choose $\til{\rho} \in (1,\rho]$ small enough so that for all $1\le s\le t\le \til\rho$, we have $|x|\le \Con^{-1}$ and $\frac12\mu s-\frac14\mu-x\mu \ge \frac18\mu$. This implies
		\begin{align*}
			\mbox{r.h.s.~of \eqref{e.l.maxbound4}} & \le 2\Pr\left(|Z_s^{0}(0)| \ge \tfrac18\mu\right)+\mathbb{P}\Big(|Z_s (x\mu; \h_0)|\geq \tfrac{1}{4}\mu\Big)+\mathbb{P}\Big(|Z_s (-x\mu; \h_0)|\geq \tfrac{1}{4}\mu\Big) \\ & \le \Con\exp\left(-\tfrac1\Con k^3\right),
		\end{align*}
		proving \eqref{eq:claim2}. This completes the proof.
	\end{proof}

	\section{Probablistic properties of the KPZ fixed point} \label{sec4}
	
	In this section, we investigate several probabilistic properties of the KPZ fixed point $\h_t(x)$ started from an initial data $\h_0$ in the ST-class. We divide this section into three subsections. Section \ref{sec4.1} studies the growth of the spatial profile $\h_t(\cdot)$ for $t$ near $1$. Sections \ref{sec4.2} and \ref{sec4.3} contain proof of several estimates related to spatial and temporal modulus of continuity for the KPZ fixed point.
	
	\subsection{Growth of the spatial profile}\label{sec4.1}
	
	The main result of this section is Proposition~\ref{prop:gc} which studies the growth of spatial profile $\h_t(x)$ for $t$ near $1$ for initial data in ST-class. 
	
	\begin{proposition}[Growth Control]\label{prop:gc} Fix any $\stp\in \R\times \R_{>0}^2$ and
		consider the KPZ fixed point $\h$ started from an initial data $\h_0$ from the \st \ defined in Definition \ref{def:st}. There exist $\tau= \tau (\stp)>0$ and $\Con=\Con(\stp)>0$ such that for all $M\ge1, s>0,$ and $t\in[1,1+\tau]$,
		\begin{align} \label{eq:gc}
			\Pr\Big(\sup_{|x|\le M} \h_{t}(x) \ge f(s,\h_0,M)+s\Big)\le M\cdot \se{3/2} .
		\end{align}
		where 
		\begin{align*}
			f(s,\h_0,M)= \begin{cases}
				0 & \text{ when } \h_0(x)=-\infty\cdot\ind_{x=0}\\
				\af+1 & \text{ when } \h_0 \mbox{ is functional and }\dt>1\\
				\af+ \frac{(1-\dt)}{\dt^2}M^2+1 & \text{ when } \h_0 \mbox{ is functional and }\dt \in (0,1)\\
				s\sqrt{M} & \text{ when } \h_0(x)=\bm(x).
			\end{cases}
		\end{align*}
	\end{proposition}

	To prove Proposition \ref{prop:gc} for Brownian initial data we need a technical estimate which we state below.
	
	\begin{lemma}\label{b_bb}  There exist universal constants $\Con>0$ such that for all $r,s>0$ we have
		\begin{align*}
			\Pr\left(\sup_{y\in \R, |z|\le r} \left[ \B(y)-\B(z)-\tfrac{(y-z)^2}{4}\right] \ge s\right) \le (r+1)\cdot \se{3/2}
		\end{align*}
		where $\B(x)$ is a two-sided Brownian motion with diffusion coefficient $2$.
	\end{lemma}

	The proof of the above Lemma follows from standard Brownian calculations and hence deferred to the appendix.
	
	\begin{proof}[Proof of Proposition \ref{prop:gc}] We set $\tau:=\tfrac12\min(\dt,1)$. We start with noting that $\sup_{|x|\le M} \h_v(x)$ is bounded above by $X_1+X_2$ where 
		\begin{align*}
			X_1:= \sup_{|x|\le M,z\in \R}\Big[\h_0(z)-\tfrac{(z-x)^2}{1+2\tau}\Big], &\quad  X_2:=\sup_{|x|\le M,z\in \R}\Big[\cl(z,0;x,t)+\tfrac{(z-x)^2}{1+2\tau}\Big].
		\end{align*}	
		By the union bound, we write 
		\begin{equation}\label{eq:2parts}
			\Pr\Big(\sup_{|x|\le M} \h_{t}(x) \ge f(s,\h_0,M)+s\Big)\leq \Pr(X_1\geq f(s,\h_0, M))+ \Pr(X_2\geq s).
		\end{equation}
		We now proceed to bound $\Pr(X_2\geq s)$. Observe that if $\delta:=\frac{\tau}{1+\tau}$, then, we have $t(1+\delta)\le 1+2\tau$ for any $t\in [1,1+\tau]$. Applying Lemma \ref{lem:bbv} for any $t\in [1,1+\tau]$ shows 
		\begin{equation}
			\label{eq:dlpart}
			\Pr(X_2\geq s)  \le \Pr\Big(\sup_{|x|\le M,z\in \R}\big[\cl(z,0;x,t)+\tfrac{(z-x)^2}{t(1+\delta)}\big] \ge s\Big)  \le M\cdot \Con \exp(-\tfrac1{\Con}s^{3/2}).
		\end{equation}
		Let us now consider the first term in r.h.s.~of \eqref{eq:2parts}. Note that $X_1$ is deterministic when $\h_0$ is deterministic. Thus the first term in the r.h.s. of \eqref{eq:2parts} will either be $1$ or $0$. When $\h_0$ is narrow wedge initial data, the first term is trivially zero. When $\h_0(x)$ is functional data with $\dt>1$, $X_1$ is atmost $A$ and hence, the first term of the r.h.s. of \eqref{eq:2parts} is zero. When $\dt\in (0,1)$ for all $z\in \R$ and $|x|\le M$, we have 
		\begin{align*}
			\h_0(z)-\tfrac{(z-x)^2}{1+\dt} \le A-\frac{\dt^2z^2-2zx+x^2}{1+\dt} \le A+\frac{1-\dt}{\dt^2}x^2 \le A+\frac{1-\dt}{\dt^2}M^2 < f(s,\h_0,M).
		\end{align*}
		This shows once again that $\Pr(X_1\geq f(s,\h_0, M))$ is zero when $\dt\in (0,1)$. This completes the proof of the proposition for non-random data. 
		
		\medskip
		
		Now we consider the case when $\h_0(z)=\bm(z)$. Set $\gamma:=\frac{1+2\tau}{4}$. By the scaling invariance (i.e., $\mathfrak{B}_{r^2x}\stackrel{d}{=} r \mathfrak{B}_x$) of the law of Brownian motion, we have
		\begin{align}\label{eq:inv}
			\sup_{|x|\le M,z\in \R}\Big[\bm(z)-\frac{(z-x)^2}{4\gamma}\Big]\stackrel{d}{=} \gamma^{1/3}\sup_{|x|\le M\gamma^{-2/3},z\in \R}\Big[\bm(z)-\frac{(z-x)^2}{4}\Big]. 
		\end{align}
		By the triangle inequality,
		\begin{align*}
			\sup_{|x|\le M\gamma^{-2/3},z\in \R}\big[\bm({z})-\frac{(z-x)^2}{4}\big] \le \sup_{|x|\le M\gamma^{-2/3},z\in \R}\big[\bm(z)-\bm(x)-\frac{(z-x)^2}{4}\big]+\sup_{|x|\le M\gamma^{-2/3}} \bm(x).
		\end{align*}

		Note that $\gamma\in[\frac14,\frac12]$. Applying Lemma \ref{b_bb} with $r\mapsto M\gamma^{-2/3}$ and $s\mapsto \frac{s}{2\gamma^{1/3}}\sqrt{M}$, we get
		\begin{align*}
			\Pr\Big(\gamma^{1/3}\sup_{|x|\le M\gamma^{-2/3},z\in \R}\big[\bm(z)-\bm(x)-\tfrac{(z-x)^2}{4}\big] \ge \tfrac{s}{2}\sqrt{M}\Big) & \le \Con (M+1)\exp(-\tfrac1{\Con}s^{3/2}M^{3/4}) \\ & \le \Con \exp(-\tfrac1{\Con}s^{3/2}),
		\end{align*}
		and by the reflection principle of Brownian motion, we have 
		\begin{align*}
			\Pr\Big(\gamma^{1/3} \sup_{|x|\le M\gamma^{-2/3}} \bm({x}) \ge \tfrac{s}{2}\sqrt{M}\Big) \le \Con\exp(-\tfrac1{\Con}s^{2}).
		\end{align*}
		Thus, by the union bound, in view of \eqref{eq:inv} and using the last two probability estimates, we get
		\begin{align*}
			\Pr\Big(\sup_{|x|\le M,z\in \R}\big[\bm(z)-\tfrac{(z-x)^2}{1+\tau}\big] \ge s\sqrt{M}\Big) \le \se{3/2}.
		\end{align*}
		Note that this bounds the first term on the right hand side of \eqref{eq:2parts}. Combining this with \eqref{eq:dlpart}, we get \eqref{eq:gc} for Brownian initial data completing the proof.
	\end{proof}

	\subsection{Spatial modulus of continuity}\label{sec4.2}
	
	The main goal of this section is to investigate the spatial modulus of continuity of the KPZ fixed point: Proposition \ref{prop:mc}. This requires a detailed study of the tail
	probabilities for difference of the KPZ fixed point at two distinct spatial point. This is done in Proposition \ref{prop:tail_modulus} below. Proposition \ref{prop:mc} then follows from Proposition \ref{prop:tail_modulus} by standard analysis.

	\begin{proposition}\label{prop:tail_modulus} Fix any $\stp\in \R\times \R_{>0}^2$ and
		consider the KPZ fixed point $\h$ started from an initial data $\h_0$ from the \st \ defined in Definition \ref{def:st}. There exists a constant $\Con(\stp)>0$ such that for all $x \neq z \in [-1, 1]$,
		\begin{equation}\label{eq:tail_modulus}
			\mathbb{P}\bigg(\frac{|\mathfrak{h}_1 (x) - \mathfrak{h}_1 (z)|}{\sqrt{|x-z|\log \frac4{|x-z|}}} \geq s\bigg) \leq \se{3/2}.
		\end{equation}
	\end{proposition}
	\begin{proof} Fix any $\h_0$ in \st. Let $v:=\sup_{r\in (0,2]} \sqrt{r\cdot \log^{-1} \frac4r}$. Assume $s$ is large enough so that $s-2v(\sqrt{s}+1)\ge \frac12s$.  Let us set
		\begin{align*}
			\til{\h}_1(\cdot):=\sup_{|y|\le \sqrt{s}} (\h_0(y)+\cl(y,0,\cdot,1)).
		\end{align*}
		We first establish the statement of the Proposition for $\til{\h}_1(\cdot)$.  
		Observe that
		\begin{align*}
			\big|\til{\h}_1 (x) - \til{\h}_1 (z)\big| & \leq \sup_{|y| \le \sqrt{s}} \left|\cl(y,0,x,1)-\cl(y,0,z,1)\right|\\
			&\leq \sup_{|y| \le \sqrt{s}}\left|\ck(y,0,x,1)-(x-y)^2-\ck(y,0,z,1)+(y-z)^2\right|\\ 
			&\leq \sup_{|y| \le \sqrt{s}}\left|\ck(y,0,x,1)-\ck(y,0,z,1)\right| + 2 |x-z| (\sqrt{s}+1). 
		\end{align*}
		By the choice of $s$ we thus have
		\begin{align}
			\notag & \mathbb{P}\bigg(\frac{|\til{\h}_1 (x) - \til{\h}_1 (z)|}{|x-z|^{\frac{1}{2}} (\log |x-z|^{-1})^{\frac{1}{2}}} \geq s\bigg) \\ \notag & \le \Pr\left(\sup_{|y| \le \sqrt{s}}\frac{\left|\ck(y,0,x,1)-\ck(y,0,z,1)\right| + 2 |x-z| (\sqrt{s} + 1)}{|x-z|^{\frac{1}{2}} (\log |x-z|^{-1})^{\frac{1}{2}}}\ge s\right) \\ & \le \Pr\left(\sup_{|y| \le \sqrt{s}}\frac{\left|\ck(y,0,x,1)-\ck(y,0,z,1)\right| }{|x-z|^{\frac{1}{2}} (\log |x-z|^{-1})^{\frac{1}{2}}}\ge \tfrac12s\right). \label{eq:twosdif}
		\end{align}
		By \eqref{eq:ProbBd3} we have
		\begin{align}
			\mbox{r.h.s.~of \eqref{eq:twosdif}} \le \Con \sqrt{s}\exp\left(-\tfrac1\Con s^{2}\right) \le \se{3/2}. \label{eq:twosdif2}
		\end{align}
		This proves \eqref{eq:tail_modulus} for $\til{\h}_1(\cdot)$. To extend it to $\h_1(\cdot)$ we utilize the estimate from Section \ref{sec3} for the argmax location of the variational problem of the KPZ fixed point. Recall $Z_1 (\cdot ;\h_0)$ from \eqref{def:ztx}.  By union bound
		\begin{align*}
			\mathbb{P}\bigg(\frac{|\mathfrak{h}_1 (x) - \mathfrak{h}_1 (z)|}{|x-z|^{\frac{1}{2}} (\log \frac4{|x-z|})^{\frac{1}{2}}} \geq s\bigg) & \leq \mathbb{P}\big(|Z_1 (x)| \geq \sqrt{s}\big) + \mathbb{P}\big(|Z_1 (z)| \geq \sqrt{s}\big)+ \mathbb{P}\bigg(\frac{|\til{\h}_1 (x) - \til{\h}_1 (z)|}{|x-z|^{\frac{1}{2}} (\log \frac4{|x-z|})^{\frac{1}{2}}} \geq s\bigg).
		\end{align*}
		Applying Corollary \ref{lem:exitbound} with $x \mapsto x s^{-1/2}, r\mapsto \sqrt{s}$ and $x \mapsto z s^{-1/2}, r\mapsto \sqrt{s}$ we know $\mathbb{P}\big(|Z_1 (x)| \geq \sqrt{s}\big)$ and $\mathbb{P}\big(|Z_1 (z)| \geq \sqrt{s}\big)$ are at most $\se{3/2}$. Appealing to these estimates and the bound from \eqref{eq:twosdif2} we see that the r.h.s.~of the equation is at most $\se{3/2}$. This concludes the proof.
	\end{proof}

	\begin{proposition}\label{prop:mc} 
		Fix any $\stp\in \R\times \R_{>0}^2$ and
		consider the KPZ fixed point $\h$ started from an initial data $\h_0$ from the \st \ defined in Definition \ref{def:st}. There exists a constant $\Con(\stp)>0$ such that 
		\begin{equation*}
			\mathbb{P}\Big(\sup_{x, z \in [-1,1]} \frac{|\sh_1 (x) - \sh_1 (z)|}{\sqrt{|x-z|} \cdot \log^2\frac{4}{|x-z|}} \ge s\Big) \leq \se{3/2}.
		\end{equation*}
	\end{proposition}
	\begin{proof}
		The proof proceeds by mimicking Levy's proof of the modulus of continuity. Set $\Lambda_n := \{2^{-n} i: i \in [-2^n,2^n-1]\cap \Z\}$ and define
		$$X_n := \sup_{x \in \Lambda_n} |\sh_1(x+2^{-n}) - \sh_1(x)|, \quad \norm{X}:=\sup_{n\ge 0}  \frac{X_n\cdot 2^{n/2}}{(n+2)^2}.$$ 
		For any $-1\le x < z\le 1$ observe that the following string of inequalities  holds deterministically
		\begin{align*}
			|\h_1(x)-\h_1(z)| & \le \sum_{n=0}^{\infty}  \left|\h_1({2^{-n}\lfloor {2^nx}\rfloor})-\h_1({2^{-n+1}\lfloor {2^{n-1}x}\rfloor})+\h_1({2^{-n}\lfloor {2^nz}\rfloor})-\h_1({2^{-n+1}\lfloor {2^nz}\rfloor})\right| \\ & \le \sum_{n=0}^{\infty} 2\left({|x-z|}2^n\wedge 2\right)X_n \\ & \le \norm{X}\sum_{n=0}^{\infty} (n+1)^22^{-n/2}\left({|x-z|}2^n\wedge 2\right) \le \theta\norm{X} \cdot |x-z|^{1/2}\log^2 \tfrac{4}{|x-z|}.
		\end{align*}
		where $\theta>0$ is an absolute constant. Hence,
		\begin{align}
			\Pr\left(\sup_{x, z \in [-1,1]} \frac{|\sh_1 (x) - \sh_1 (z)|}{\sqrt{|x-z|} \cdot \log^2\frac{4}{|x-z|}} \ge s \right) \le \Pr(\norm{X} \ge \tfrac1\theta s). \label{eq:mod}
		\end{align}
		However by union bound
		\begin{align*}
			\Pr\left(\norm{X} \ge s\right) & \le \sum_{n=0}^{\infty}\sum_{x\in \Lambda_n} \Pr\left(|\sh_1(x+2^{-n}) - \sh_1(x)| \ge s 2^{-n/2}(n+2)^2\right) \\ & \le \sum_{n=0}^{\infty} \Con \cdot 2^{n+1}\exp\left(-\tfrac1\Con s^{3/2} (n+2)^{9/4}\right) \le \se{3/2},
		\end{align*}
		where the penultimate inequality follows by applying Proposition \ref{prop:tail_modulus} to each increment. Adjusting the $\Con>0$ further we see that the above display implies $\mbox{r.h.s.~of \eqref{eq:mod}} \le \se{3/2}$. This completes the proof.
	\end{proof}

	\subsection{Temporal Modulus of Continuity} \label{sec4.3} We now study the KPZ fixed point in the temporal direction with spatial location fixed at $x=0$.
	
	\begin{proposition}\label{prop:stincrementtail} Fix any $\stp\in \R\times \R_{>0}^2$ and
		consider the KPZ fixed point $\h$ started from an initial data $\h_0$ from the \st \ defined in Definition \ref{def:st}. There exist constants $\rho(\stp)>1$ and $\Con(\stp)>0$ such that for all $1 \leq s < t \leq \rho$, we have 
		\begin{equation*}
			\P\bigg(\frac{|\h_t(0) - \h_{s} (0)|}{{(t - s)}^{\frac{1}{3}}} \geq r\bigg) \leq \Con \exp\big(-(\tfrac{2}{3} - \e) r^\frac{3}{2}\big).
		\end{equation*}
	\end{proposition}
	\begin{proof} By the variational formula for the KPZ fixed point we know $$\frac{\h_t(0) - \h_{s}(0)}{(t-s)^{\frac{1}{3}}} \ge (t-s)^{-1/3}\L(0,0;0,t-s) \stackrel{d}{=} \tw.$$ By the lower tail estimates for  $\tw$ from \cite{ramirez2011beta} it follows that
		\begin{equation*}
			\P\Big(\frac{\h_t(0) - \h_{s}(0)}{(t-s)^{\frac{1}{3}}} \leq -r\Big) \leq \Con\exp\left(-\tfrac{1}{24}r^3\right). 
		\end{equation*}
		Hence it suffices to show that
		for $1 \leq s < t \leq \rho$ and $\rho-1$ small enough,
		\begin{equation*}
			\P\Big(\frac{\h_t(0) - \h_{s}(0)}{(t-s)^{\frac{1}{3}}} \geq r\Big) \leq \Con \exp(-(\tfrac{2}{3} - \e) r^{\frac{3}{2}}\Big).
		\end{equation*}

		Without loss of generality we may assume that $r > 1$ throughout the proof. Fix any $\e>0$ and $\stp\in \R\times \R_{>0}^2$. Consider any $\h_0$ in \st.  Let us pick $\rho(\stp)>1$, $\Con(\stp)>0$ from Lemma \ref{lem:maxbound}. Set $\eta=\eta(\stp)>0$ such that $\frac1\Con \eta^3=\frac23$. By Lemma \ref{lem:maxbound} we have
		\begin{align}
			\notag
			&\P\bigg(\sup_{x\in \R}\Big(\frac{\h_{s}((t-s)^{\frac{2}{3}}x) - \h_{s}(0)}{(t-s)^{\frac{1}{3}}} + \calA(x)-x^2 \Big) \geq r\bigg)\\
			&\label{e.p.sttail1}
			\leq \P\bigg(\sup_{|x|\le \eta \sqrt{r}}\Big(\frac{\h_{s}((t-s)^{\frac{2}{3}}x) - \h_{s}(0)}{(t-s)^{\frac{1}{3}}} + \calA(x)-x^2\Big) \geq  r\bigg) + \Con\exp\left(-\tfrac23 r^\frac{3}{2}\right).
		\end{align}
		Define the event in the first term of \eqref{e.p.sttail1} to be 
		\begin{equation*}
			\AA := \bigg\{ \sup_{|x|\le \eta \sqrt{r}}\Big(\frac{\h_{s}((t-s)^{\frac{2}{3}}x) - \h_{s}(0)}{(t-s)^{\frac{1}{3}}} + \calA(x)-x^2\Big) \geq  r\bigg\}.
		\end{equation*}
		We set $\mu=\sqrt{r}(t-s)^{-\frac14}$. Recall the event $\EE_t^{\mu}(a)$ from \eqref{eq:etma}.  By Lemma \ref{lem:argmaxprop} \ref{lem:compare}, on $\EE_s^{\mu}
		((t-s)^{\frac{2}{3}}\eta\sqrt{r})$, we have 
		\begin{align*}
			\h_{s}(x) - \h_{s} (0) &\leq \h^{\mu}_{s} (x) - \h^{\mu}_{s} (0), \qquad x \in [0, (t-s)^{\frac{2}{3}}\eta\sqrt{r}],\\
			\h_{s} (x) - \h_{s} (0) &\leq \h_{s}^{-\mu} (x) - \h_{s}^{-\mu} (0),\quad\, x \in [-(t-s)^{\frac{2}{3}}\eta \sqrt{r}, 0].
		\end{align*}
		Consequently, 
		\begin{align}\notag
			\P\big(\AA\big) &\leq \P\big(\AA\, \cap\, \EE_s^{\mu}
			((t-s)^{\frac{2}{3}}\eta\sqrt{r})\big) + \P\big(\EE_s^{\mu}
			((t-s)^{\frac{2}{3}}\eta\sqrt{r})^c\big)\\ 
			\label{e.p.sttail2}
			&\leq \P\big(\DD_{\mu} \big) + \P\big(\EE_s^{\mu}
			((t-s)^{\frac{2}{3}}\eta\sqrt{r})^c\big).
		\end{align}
		where 
		\begin{align*}
			\DD_{\mu} & := \Bigg\{\sup_{x\in [0, \eta\sqrt{r}]} \Big(\frac{\h^\mu_{s}((t-s)^{\frac{2}{3}}x) - \h^\mu_{s}(0)}{(t-s)^{\frac{1}{3}}} + \calA(x)-x^2\Big) \geq r, \text{ and } 
			\\ 
			& \hspace{2cm}  \sup_{x\in [-\eta \sqrt{r}, 0]}\Big(\frac{\h^{-\mu}_{s}((t-s)^{\frac{2}{3}}x) - \h^{-\mu}_{s}(0)}{(t-s)^{\frac{1}{3}}} + \calA(x)-x^2 \Big) \geq r\Bigg\}.
		\end{align*}
		We now claim that one can choose $\til{\rho}$ sufficiently close to $1$ such that for all $1\le s<t \le \til\rho$
		\begin{equation}
			\label{eq:claim3}
			\Pr(\DD_\mu) \le \Con\exp\left(-(\tfrac23-\e)r^{\frac32}\right),  \quad \P\big(\EE_s^{\mu}
			((t-s)^{\frac{2}{3}}\eta\sqrt{r})^c\big)\le \Con\exp\left(-\tfrac23r^{\frac32}\right).
		\end{equation}
		Plugging this bound to \eqref{e.p.sttail2}, in view of \eqref{e.p.sttail1} we get the desired bound. We now proceed to show \eqref{eq:claim3} in the following two steps.
		
		\medskip
		
		\noindent\textbf{Step 1: Upper bound of $\P\big(\DD_{\mu}\big)$.}
		We use $\BR_i$ to represent a random variable with Baik Rains distribution.  By Lemma \ref{lem:2.3} \ref{2.3c} We have
		\begin{align*}
			\sup_{x \in [0, \eta\sqrt{r}]}\Big(\frac{\h^\mu_{s}((t-s)^{\frac{2}{3}}x) - \h^\mu_{s}(0)}{(t-s)^{\frac{1}{3}}} + \calA(x)-x^2\Big)
			&\overset{d}{=} \sup_{x \in [0, \eta\sqrt{r}]}\Big(\B(x)+\mu(t-s)^{\frac13} x + \calA(x)-x^2\Big) \\ & \leq \BR_1 + (t-s)^{\frac{1}{3}} \mu \eta \sqrt{r}. 
		\end{align*}
		where $\BR_1$ denotes a random variable with Baik-Rains distribution \cite{br}. Similarly,
		\begin{align*}
			\sup_{x \in [-\eta\sqrt{r},0]}\Big(\frac{\h^{-\mu}_{s}((t-s)^{\frac{2}{3}}x) - \h^{-\mu}_{s}(0)}{(t-s)^{\frac{1}{3}}} + \calA(x)-x^2\Big)
			&\leq \BR_2 + (t-s)^{\frac{1}{3}} \mu \eta \sqrt{r}.
		\end{align*}
		for another random variable $\BR_2$ following Baik-Rains distribution. The correlation between $\BR_1,\BR_2$ is not important since we will only use the one point tail bound.
		Since $\mu = r^{\frac{1}{2}} (t-s)^{-\frac{1}{4}}$, applying an union bound along with Lemma \ref{lem:onepttail} \ref{2.6c} we have
		\begin{align}
			\notag
			\P\big(\DD_{\mu}\big) &\leq \P\Big(\max\big(\BR_1, \BR_2\big) + (t-s)^{\frac{1}{3}} \mu \eta \sqrt{r} \geq r\Big) 
			\\
			&\leq 
			\sum_{i=1}^2\P\Big(\BR_i\geq r[1 - \eta (t-s)^{\frac{1}{12}}]\Big) \leq \Con \exp\big(-(\tfrac{2}{3} - \e) r^{\frac{3}{2}}\big), \notag
		\end{align}
		for $t - s$ small enough. This proves the first inequality in \eqref{eq:claim3}.
		
		\medskip
		
		\noindent\textbf{Step 2: Upper bound of $\P\big(\EE_s^{\mu}
			((t-s)^{\frac{2}{3}}\eta\sqrt{r})^c\big)$.} 
		Set $a=(t-s)^{\frac{2}{3}}\eta\sqrt{r}$. By same argument as in \eqref{e.l.maxbound4} we have 
		\begin{equation}
			\label{45}
			\begin{aligned}
				\P\big(\EE_s^{\mu}
				(a)^c\big) & \leq \mathbb{P}\Big(Z^{0}_s(0) -a+\tfrac12\mu s \leq \tfrac14\mu\Big) +  \mathbb{P}\Big(Z_s (a; \h_0)\geq \tfrac{1}{4}\mu\Big) \\
				& \hspace{2cm}+\mathbb{P}\Big(Z^{0}_s(0) +a-\tfrac12\mu s \geq -\tfrac14\mu\Big) +  \mathbb{P}\Big(Z_s (-a; \h_0)\leq -\tfrac{1}{4}\mu\Big).
			\end{aligned}
		\end{equation}
		Recall that $\mu = r^{\frac{1}{2}} (t-s)^{-\frac{1}{4}}$. Setting $x=a/\mu$ we have $|x|\le (t-s)^{11/12}\eta$. One can take $t-s$ small enough so that Corollary \ref{lem:exitbound} becomes applicable with $r\mapsto \frac14\mu$. Applying Corollary \ref{lem:exitbound} then yields
		\begin{align*}
			\mbox{r.h.s.~of \eqref{45}} \le \Con \exp\left(-\tfrac1\Con \mu^3\right)=\Con \exp\left(-\tfrac1\Con (t-s)^{-3/4}r^{\frac32}\right).
		\end{align*}
		Choosing $t,s$ close enough so that $(t-s)^{-3/4} \ge \frac23$, we thus arrive at the second inequality in \eqref{eq:claim3}. This completes the proof. 
	\end{proof}
	Appealing to Lemma 3.3 in \cite{dv21} we have the following immediate corollary.
	\begin{corollary}\label{cor:pmod} Fix any $\stp\in \R\times \R_{>0}^2$ and
		consider the KPZ fixed point $\h$ started from an initial data $\h_0$ from the \st \ defined in Definition \ref{def:st}. There exist constants $\rho(\stp)>1$ and $\Con(\stp)>0$ such that
		\begin{equation*}
			\P\bigg(\sup_{t\neq s, t,s\in [1,\rho]}\frac{|\h_t(0) - \h_{s} (0)|}{{(t - s)}^{\frac13}\log^{2/3}\frac{2}{|t-s|}} \geq r\bigg) \leq \Con \exp\big(-\tfrac1\Con r^\frac{3}{2}\big).
		\end{equation*}
	\end{corollary}
	
	We also have the following weaker version of modulus of continuity as a consequence of Proposition \ref{prop:stincrementtail}, which will be useful in proving our main results.

	\begin{proposition}\label{pp:weakc} Fix any $\stp\in \R\times \R_{>0}^2$ and
		consider the KPZ fixed point $\h$ started from an initial data $\h_0$ from the \st \ defined in Definition \ref{def:st}. There exist constants $\rho(\stp)>1$ and $\Con(\stp)>0$ such that for all $1\le a\le b\le \rho$ we have
		\begin{equation*}
			\P\bigg(\sup_{t \in [a, b]}\frac{|\h_t (0) - \h_a(0)|}{(b-a)^{\frac{1}{3}}} \geq s\bigg) \leq \se{3/2}.
		\end{equation*}
	\end{proposition}
	
	\begin{proof} Let us denote $\h_t:=\h_t(0)$. Consider $t_{n, i} = a + (b - a) i 2^{-n}$, $i =0, 1, \dots, 2^n$ and define
		\begin{equation*}
			X_n = \sup_{i = 1, \dots, 2^n} |\h_{t_{n, i}} - \h_{t_{n, i-1}}|.
		\end{equation*} Note that for each $t\in [a,b]$ we have
		$$|\h_t-\h_a|=\left|\sum_{n=1}^{\infty} \h_{v_n}-\h_{v_{n-1}}\right| \le \sum_{n=1}^{\infty}\left|\h_{v_n}-\h_{v_{n-1}}\right| \le \sum_{n=1}^{\infty} X_n$$ 
		where $v_n:=a+ (b-a)2^{-n}\lfloor (t-a)2^n/(b-a)\rfloor$. Fix any $\theta\in (0,1)$ By union bound,
		\begin{align*}
			\Pr\left(\sup_{t\in [a,b]} \frac{|\h_t-\h_a|}{(b-a)^{\frac13}} \ge s\right) & \le \sum_{n=1}^{\infty} \Pr\left(  \frac{X_n}{(b-a)^{\frac13}} \ge (1-\theta)\theta^n s\right) \\ & \le \sum_{n=1}^{\infty} \sum_{i=1}^{2^n} \Pr\left(  \frac{|\h_{t_{n, i}} - \h_{t_{n, i-1}}|}{2^{-\frac{n}3}(b-a)^{\frac13}} \ge (1-\theta)2^{\frac{n}3}\theta^n s\right) \\ & \le \Con \sum_{n=1}^{\infty} 2^n\exp\left(-\tfrac1\Con \left[(1-\theta)2^{n/3}\theta^n s\right]^{3/2} \right) \le \se{3/2}. 
		\end{align*}
		Set $\theta=2^{-1/4}$ so that $2^{n/3}\theta^n =2^{n/12}$. This forces the right side of the above equation to be at most $\se{3/2}$, concluding the proof.
	\end{proof}

	\section{Landscape Replacements}\label{sec5}
	
	In this section we describe how to extract independent structure within the KPZ fixed point. The key idea is the novel construction of proxies for the KPZ fixed point which we termed as `landscape replacement'. Loosely speaking we `replace' sections of the directed landscape in the variational problem of the KPZ fixed point to obtain proxies for the KPZ fixed point at certain time points. Our construction of proxies is different for long time and short time cases. In Section \ref{sec5.1} we describe our long time landscape replacement approach whereas Section \ref{sec5.2} discusses short time landscape replacement. The results within these sections, namely Theorem \ref{ind} and Theorem \ref{thm:lr} form the important components for the lower bound of long time and short time LIL respectively.

	\subsection{Long time Landscape Replacement} \label{sec5.1}
	
	In this section we prove our long time landscape replacement theorem. Before going into the details of the theorem  we first describe how any initial data in \lt \ can be identified with initial data from \st. This allows us to apply all the estimates from the previous sections that pertains to initial data within the \st.

	\medskip
	
	Fix any $\h_0$ from the \lt \ for the rest of this section and let $\h_0(0)=\bet$.  Observe that for any $v\ge 1$, $\til{\h}_0^{(v)}(x):=v^{-1}\h_0(v^2x)$ is still within the \lt \ with the same constant $\csq$, and if $|\h_0(0)|\le \bet$, we have $|\til{\h}_0^{(v)}(0)|\le \bet$. Furthermore for any functional data within the \lt, one can find $\af(\csq)$ such that $\h_0(x)\le \csq\sqrt{1+|x|}\le \af+\frac12x^2$. In this way, $\h_0$ and $\til\h^{(v)}_0$ (for each $v\ge 1$) can be viewed as data from \st \ with $\stp=(\af,\frac12,\bet)$ with $\af$ as identified as above. 
	
	\medskip
	
	Let us now turn toward our long time landscape replacement result. We define
	\begin{align}\label{long}
		H^{t_2\downarrow t_1}:=\sup_{z\in \R} \left\{\h_0(z)+\cl(z,t_1;0,t_2)\right\}.
	\end{align}
	Note that $\h_{at}(0)$ and $H^{(a+at)\downarrow a}$ are the same in distribution for each $a$ and $t$. We call $H^{(a+at)\downarrow a}$ to be the proxy for $\h_{a+at}(0)$ for the following result.
	\begin{theorem}[Long time landscape replacement]\label{ind} Fix any $\ltp\in \R_{>0}^2$ and
		consider the KPZ fixed point $\h$ started from an initial data $\h_0$ from the \lt \ defined in Definition \ref{def:lt}. There exists a constant $\Con=\Con(\ltp)>0$ such that for all $a,t\ge 1$ we have
		\begin{align*}
			\Pr\left(\tfrac1{a^{1/3}t^{1/3}}\left|\h_{a+at}(0)-H^{(a+at)\downarrow a}\right| \ge 1\right) \le \Con\exp({-\tfrac1\Con t^{1/3}}).
		\end{align*}
	\end{theorem}

	\begin{proof} It suffices to prove the result for large enough $t$. Set $\mathfrak{g}_0(x)=(at)^{-1/3}\til{\h}_0((at)^{2/3}x).$ We may write
		$$\tfrac1{a^{1/3}t^{1/3}}\left|\h_{a+at}(0)-H^{(a+at)\downarrow a}\right|=\left|\mathfrak{g}_{1+t^{-1}}-G^{(1+t^{-1})\downarrow t^{-1}}\right|,$$
		where 
		\begin{align*}
			\mathfrak{g}_{1+t^{-1}}: =\sup_{z\in \R} [\g_0(z)+\cl(z,0;0,1+t^{-1})], \quad G^{(1+t^{-1})\downarrow t^{-1}}:=\sup_{z\in \R} [\g_0(z)+\cl(z,t^{-1};0,1+t^{-1})].
		\end{align*}
		Let us define
		\begin{align*}
			\til{\mathfrak{g}}_{1+t^{-1}}: =\sup_{|z|\le t^{1/6}} [\g_0(z)+\cl(z,0;0,1+t^{-1})], \quad \til{G}^{(1+t^{-1})\downarrow t^{-1}}:=\sup_{|z|\le t^{1/6}} [\g_0(z)+\cl(z,t^{-1};0,1+t^{-1})].
		\end{align*}
		Although the initial data $\g_0$ depends on $a$ and $t$, as explained in the begining of this section one can find $\stp\in \mathbb{R}\times\R_{>0}^2$ depending on $\ltp$ such that $\g_0$ is always within \st, for all $a,t\ge 1$. We may thus apply all our results from previous sections that are based on \st. We apply Corollary \ref{lem:exitbound} with $x\mapsto 0,r \mapsto t^{1/6},t\mapsto 1$ and $x\mapsto 0,r \mapsto t^{1/6},t\mapsto 1+t^{-1}$, we see that with probability $1-\exp\left(-\tfrac1\Con t^{1/2}\right)$, $\mathfrak{g}_{1+t^{-1}}=\til{\mathfrak{g}}_{1+t^{-1}}$ and ${G}^{(1+t^{-1})\downarrow t^{-1}}=\til{G}^{(1+t^{-1})\downarrow t^{-1}}$ respectively (see also Remark \ref{rm:anymax}). Thus,
		\begin{align*}
			& \Pr\left(\left|\mathfrak{g}_{1+t^{-1}}-G^{(1+t^{-1})\downarrow t^{-1}}\right| \ge 1\right)\le \Pr\left(\left|\til{\mathfrak{g}}_{1+t^{-1}}-\til{G}^{(1+t^{-1})\downarrow t^{-1}}\right| \ge 1\right)+\exp\left(-\tfrac1\Con t^{1/2}\right).
		\end{align*}
		However,
		\begin{align*}
			\left|\til{\mathfrak{g}}_{1+t^{-1}}-\til{G}^{(1+t^{-1})\downarrow t^{-1}}\right|\le \sup_{|z|\le t^{1/6}}\left| \ck(z,0;0,1+t^{-1})-\ck(z,t^{-1};0,1+t^{-1})\right|+\tfrac{t^{1/3}}{1+t}.
		\end{align*}
		Taking $t$ large enough we see that
		\begin{align*}
			& \Pr\left(\left|\til{\mathfrak{g}}_{1+t^{-1}}-\til{G}^{(1+t^{-1})\downarrow t^{-1}}\right| \ge 1\right)\\ & \le \Pr\left(\sup_{|z|\le t^{1/6}}t^{2/9}\left| \ck(z,0;0,1+t^{-1})-\ck(z,t^{-1};0,1+t^{-1})\right|\ge \tfrac12t^{2/9}\right) \\ & \le \Con t^{1/6}\exp\left(-\tfrac1\Con t^{1/3}\right) \le \til{\Con}\exp\left(-\tfrac1{\til{\Con}} t^{1/3}\right),
		\end{align*}
		where the penultimate inequality follows from modulus of continuity of $\ck$ from Proposition 10.5 in \cite{DOV18}. This concludes the proof.
	\end{proof}
	\subsection{Short time Landscape Replacement}\label{sec5.2} The main goal of this section is to prove short time landscape replacement, namely Theorem \ref{thm:lr}. Towards this end, for $t_2>t_1$ we define
	\begin{align}\label{def:dwnar}
		H_{1}^{t_2\downarrow t_1}:=\sup_{z\in \R} \left\{\h_1(z)+\cl(z,1+t_1;0,1+t_2)\right\}.
	\end{align}
	Here $H_1^{t_2\downarrow t_1}$ will be the proxy for $\h_{1+t_2}(0)$ where $t_1=(at)^{-1}$ and $t_2=a^{-1}+(at)^{-1}$ with $a,t$ large enough. Due to technical reasons we impose further conditions on the pair $(a,t)$. 
	\begin{definition}[Permissible pairs] \label{pair} We call an ordered pair of reals $(a,t)$ to be permissible if
		\begin{align*}
			t\ge 1, \quad a \ge t^{1/8}, \quad 
			\log^{8/3}(at^{7/8}) \le \frac{t^{1/6}}{2\sqrt{2}\Delta_2},
		\end{align*}
		where $\Delta_2>0$ is an absolute constant that comes from Lemma \ref{l:rai}.
	\end{definition}
	We have the following theorem about replacement of appropriate sections of Landscapes in case of short time. 
	\begin{theorem}[Short time landscape replacement] \label{thm:lr}  Fix any $\stp\in \R\times \R_{>0}^2$ and
		consider the KPZ fixed point $\h$ started from an initial data $\h_0$ from the \st \ defined in Definition \ref{def:st}. There exist a constant $\Con=\Con(\stp)>0$ such that for all permissible pairs $(a,t)$ we have
		\begin{align}\label{eq:lr}
			\Pr\left(a^{\frac13}{\left|\h_{1+a^{-1}+(at)^{-1}}(0)-H_1^{a^{-1}+(at)^{-1}\downarrow (at)^{-1}}\right|}\ge 1\right) \le \Con(at)^{\frac23}\exp\left(-\tfrac1{\Con}t^{\frac3{16}}\right).
		\end{align}	
	\end{theorem}
	The above theorem is sufficient for our purpose as in our applications of the above theorem in next section $(a,t)$ will be a permissible pair.
	\begin{proof} As usual we may assume $a,t$ are large enough. As otherwise we may choose $\Con$ appropriately large. Define
		\begin{align*}
			& \til\h_{1+a^{-1}+(at)^{-1}}(0)  :=\sup_{|y|,|z|\le 1} \left[\h_1(y)+\cl\left(y,1;z,1+\tfrac1{at}\right)+\cl\left(z,1+\tfrac1{at};0,1+\tfrac1a+\tfrac1{at}\right)\right], \\
			& \til{H}_1^{a^{-1}+(at)^{-1}\downarrow (at)^{-1}} := \sup_{|z|\le 1} \left[\h_1(z)+\cl\left(z,1+\tfrac1{at};0,1+\tfrac1a+\tfrac1{at}\right)\right].
		\end{align*}
		
		For clarity we divide the rest of the proof into four steps. 
		
		\medskip
		
		\noindent\textbf{Step 1.} In this step we show
		\begin{align}\label{s1}
			\Pr\left({H}_1^{a^{-1}+(at)^{-1}\downarrow (at)^{-1}}\neq \til{H}_1^{a^{-1}+(at)^{-1}\downarrow (at)^{-1}}\right) \le \Con \exp(-\tfrac1\Con (at)^{2}) \le  \Con \exp(-\tfrac1\Con t^{9/4}).
		\end{align}
		Recall that
		\begin{align*}
			{H}_1^{a^{-1}+(at)^{-1}\downarrow (at)^{-1}}=\sup_{z\in \R}\left[\h_1(z)+\cl\left(z,1+\tfrac1{at};0,1+\tfrac1a+\tfrac1{at}\right)\right].
		\end{align*}
		We write $\cl\left(z,1+\tfrac1{at};0,1+\tfrac1a+\tfrac1{at}\right)= (at)^{-1/3}\calA((at)^{2/3}z)-atz^2,$ where $\calA$ is a stationary $\operatorname{Airy}_2$ process independent of $\h_1$. Observe that 
		\begin{equation} \label{eq:trick}
			\begin{aligned}
				& \Pr\left({H}_1^{a^{-1}+(at)^{-1}\downarrow (at)^{-1}}\neq \til{H}_1^{a^{-1}+(at)^{-1}\downarrow (at)^{-1}}\right) \\ & \le \Pr\left(\sup_{|z|\ge 1} \left[\h_1(z)-\h_1(0)+\frac{\calA\left((at)^{\frac23}z\right)-(at)^{\frac43}z^2}{(at)^{\frac13}}\right] \ge (at)^{-\frac13}\calA(0)\right) \\ & \le \Pr\left(\sup_{|x|\ge (at)^{\frac23}} \left[\frac{\h_1\left((at)^{-\frac23}x\right)-\h_1(0)}{(at)^{-\frac13}}+\calA\left(x\right)-x^2\right] \ge \calA(0)\right) \le \Con\exp\left(-\tfrac1\Con (at)^{2}\right)
			\end{aligned}
		\end{equation}
		where the last inequality follows by applying Lemma \ref{lem:maxbound} assuming $at$ is large enough. This proves the first inequality in \eqref{s1}. The second one follows as $a\ge t^{1/8}$.

		\medskip
		
		\noindent\textbf{Step 2.} In this step we show
		\begin{align}\label{s2}
			\Pr\left(\h_{1+a^{-1}+(at)^{-1}}(0)\neq \til\h_{1+a^{-1}+(at)^{-1}}(0)\right) \le \Con \exp(-\tfrac1\Con a^{2}) \le \Con \exp(-\tfrac1\Con t^{1/4}).
		\end{align}
		By the metric composition law of directed landscape \eqref{metlaw} we may write
		\begin{align*}
			\h_{1+a^{-1}+(at)^{-1}}(0)  & := \sup_{z\in \R} \left[\h_{1+(at)^{-1}}(z)+\cl\left(z,1+\tfrac1{at};0,1+\tfrac1a+\tfrac1{at}\right)\right] \\ & = \sup_{y,z \in \R} \left[\h_1(y)+\cl\left(y,1;z,1+\tfrac1{at}\right)+\cl\left(z,1+\tfrac1{at};0,1+\tfrac1a+\tfrac1{at}\right)\right] \\ & = \sup_{y\in \R} \left[\h_1(y)+\cl\left(y,1;0,1+\tfrac1a+\tfrac1{at}\right)\right].
		\end{align*}
		Applying the same trick as in \eqref{eq:trick} we see that Proposition \ref{lem:maxbound} implies that with probability at least $1-\Con\exp\left(-\tfrac1\Con a^{2}\right)$, the above supremum is attained within $z\in \R, |y|\ge 1$. The same trick  in \eqref{eq:trick} also shows that due to Proposition \ref{lem:maxbound} with probability at least $1-\Con\exp\left(-\tfrac1\Con (at)^{2}\right)$, the above supremum is attained within $y\in \R, |z|\le 1$. Thus by union bound, we have the first inequality \eqref{s2}. The second one follows as $a\ge t^{1/8}$.
		
		\medskip
		
		\noindent\textbf{Step 3.} Due to \textbf{Step 1} and \textbf{Step 2}, it suffices to prove the theorem with $\h_{1+a^{-1}+(at)^{-1}}(0)$ and $H_1^{a^{-1}+(at)^{-1}\downarrow (at)^{-1}}$ replaced by $\til\h_{1+a^{-1}+(at)^{-1}}(0)$ and $\til{H}_1^{a^{-1}+(at)^{-1}\downarrow (at)^{-1}}$ respectively. Set $s=t^{1/8}$ and consider the events
		\begin{align*}
			\mathsf{B}_{a,t} & :=\left\{\sup_{|y|,|z|\le 1} \frac{|\h_1(y)-\h_1(z)|}{|y-z|^{\frac12}\log^2\frac{2}{|y-z|}} \ge s\right\}, \\
			\mathsf{C}_{a,t} & = \left\{\sup_{y\in\R,|z|\le 1} \left[\cl\left(y,1;z,1+\tfrac1{at}\right)+\tfrac{at(y-z)^2}{4}\right]\ge \tfrac{a^{-1/3}}{2}\right\}.
		\end{align*}
		On $\mathsf{B}_{a,t}^c \cap \mathsf{C}_{a,t}^c$ we have	
		\begin{align*}
			\til\h_{1+a^{-1}+(at)^{-1}}(0)  & = \sup_{|y|,|z|\le1} \left[\h_1(y)-\h_1(z)-\tfrac{at(y-z)^2}{4}+\cl\left(y,1;z,1+\tfrac1{at}\right)+\tfrac{at(y-z)^2}{4}\right. \\ & \hspace{3cm}\left.+\h_1(z)+\cl\left(z,1+\tfrac1{at};0,1+\tfrac1a+\tfrac1{at}\right)\right] \\ & \le \sup_{|y|,|z|\le 1} \left[\h_1(y)-\h_1(z)-\tfrac{at(y-z)^2}{4}\right]\\ & \hspace{2cm}+\sup_{y\in \R,|z|\le 1}\left[\cl\left(y,1;z,1+\tfrac1{at}\right)+\tfrac{at(y-z)^2}{4}\right] +\til{H}_1^{a^{-1}+(at)^{-1}\downarrow (at)^{-1}} \\ & \le \sup_{|y|,|z|\le 1} \left[s|y-z|^{\frac12}\log^2\tfrac{2}{|y-z|}-\tfrac{at(y-z)^2}{4}\right]+\tfrac12{a^{-\frac13}} +\til{H}_1^{a^{-1}+(at)^{-1}\downarrow (at)^{-1}}.
		\end{align*}  
		Note that as $s=t^{1/8}$, applying Lemma \ref{l:rai} with $x\mapsto \frac{|y-z|}{2}$, $\alpha\mapsto 2$, $\gamma \mapsto \frac{at}{s} \ge 2$ to get	
		\begin{align*}
			\sup_{|y|,|z|\le 1} \left[s|y-z|^{\frac12}\log^2\tfrac{2}{|y-z|}-\tfrac{at(y-z)^2}{4}\right] & = \sqrt{2}s\sup_{x\in[0,1]} \left[x^{\frac12}\log^2\tfrac{1}{x}-\tfrac{at}{s}x^2\right] \\ & \le \Delta_2 \sqrt{2}s(at/s)^{-\frac13}\log^{8/3}(at/s) \\ & = a^{-\frac13} \cdot \left[\Delta_2 \sqrt{2} t^{-\frac16}\log^{8/3}(at^{7/8})\right] \le \tfrac12a^{-\frac13},
		\end{align*}
		where the last inequality follows as $(a,t)$ is permissible. Thus, on $\mathsf{B}_{a,t}^c \cap \mathsf{C}_{a,t}^c$ we have
		\begin{align*}
			\til\h_{1+a^{-1}+(at)^{-1}}(0)  & \le {a^{-\frac13}} +\til{H}_1^{a^{-1}+(at)^{-1}\downarrow (at)^{-1}}.
		\end{align*}
		On the other hand, by Proposition \ref{prop:mc} we have $\Pr\left(\mathsf{B}_{a,t}\right) \le \se{3/2}=\te{3/16}$. Apply Lemma \ref{lem:bbv} with $\delta\mapsto 3, v\mapsto (at)^{-1}$, $s\mapsto \frac12a^{-\frac13}$, and $r\mapsto 1$, to get $\Pr\left(\mathsf{C}_{a,t}\right) \le \Con((at)^{\frac23}+1)\exp(-\frac1{\Con}t^{\frac12}).$ Hence,
		\begin{align}\label{s3}
			\Pr\left(\til\h_{1+a^{-1}+(at)^{-1}}(0) \ge a^{-\frac13}+\til{H}_1^{a^{-1}+(at)^{-1}\downarrow (at)^{-1}}\right) & \le \Pr\left(\mathsf{B}_{a,t}\right)+\Pr\left(\mathsf{C}_{a,t}\right) \\ & \le \Con (at)^{\frac23}\exp\left(-\tfrac1\Con t^{\frac3{16}}\right). \notag
		\end{align}

		\medskip
		
		\noindent\textbf{Step 4.} Consider the event
		\begin{align*}
			\mathsf{D}_{a,t} & :=\left\{\inf_{|z|\le 1} \cl(z,1;0,1+(at)^{-1}) \le -a^{-\frac13}\right\}.
		\end{align*}
		On $\mathsf{D}_{a,t}^c$ we have
		\begin{align*}
			\til\h_{1+a^{-1}+(at)^{-1}}(0)  & \ge \sup_{|z|\le 1} \left[\h_1(z)+\cl\left(z,1;z,1+\tfrac1{at}\right)+\cl\left(z,1+\tfrac1{at};0,1+\tfrac1a+\tfrac1{at}\right)\right] \\ & \ge \sup_{|z|\le 1} \left[\h_1(z)+\cl\left(z,1+\tfrac1{at};0,1+\tfrac1a+\tfrac1{at}\right)\right]+\inf_{|z|\le 1}\cl\left(z,1;z,1+\tfrac1{at}\right) \\ & \ge \til{H}_1^{a^{-1}+(at)^{-1}\downarrow (at)^{-1}}-{a^{-\frac13}}.
		\end{align*}
		On the other hand by union bound we have, note that by modulus of continuity of $\ck$ and tail bounds for $\ck$ we have
		\begin{align*}
			\Pr\left(\mathsf{D}_{a,t}\right) & \le \Pr\left(\sup_{|z|\le 1} |\cl(z,1;z,1+(at)^{-1})|\ge a^{-\frac13}\right) \\ & \le \Pr\left(\sup_{|z|\le (at)^{\frac23}} |\ck(z,1;z,1)|\ge t^{\frac13}\right) \\ & \le \sum_{k=-\lceil (at)^{\frac23} \rceil}^{\lceil (at)^{\frac23} \rceil} \Pr\left(\sup_{z\in [k,k+1]} |\ck(z,1;z,1)|\ge t^{\frac13}\right) \le  \Con(at)^{2/3}\exp\left(-\tfrac1\Con t^{1/2	}\right),
		\end{align*}
		where the last inequality follows \eqref{eq:ProbBd2}. Thus we have
		\begin{align}\label{s4}
			\Pr\left(\til\h_{1+a^{-1}+(at)^{-1}}(0) \ge a^{-\frac13}+\til{H}_1^{a^{-1}+(at)^{-1}\downarrow (at)^{-1}}\right) \le \Pr\left(\mathsf{D}_{a,t}\right)\le \Con(at)^{2/3}\exp\left(-\tfrac1\Con t^{1/2	}\right).
		\end{align}
		Combining \eqref{s3} and \eqref{s4}, thanks to \eqref{s1} and \eqref{s2}, we arrive at \eqref{eq:lr}.
	\end{proof}

	\section{Proof of main results}\label{sec6}
	
	In this section we prove our main results on long time and short time Law of Iterated logarithms. We prove long time and short time LIL in Section \ref{sec6.1} and Section \ref{sec6.2} respectively. 
	
	\subsection{Long Time LIL: Proof of Theorems \ref{thm:longLILa} and \ref{thm:longLILb}} \label{sec6.1}
	
	In this section we prove Theorems \ref{thm:longLILa} and \ref{thm:longLILb}. Fix $\ltp=(\csq,\bet)\in \R_{>0}^2$. Fix any $\h_0$ from the \lt. For any $v\ge 1$,  set $\til{\h}_0^{(v)}(x):=v^{-1}\h_0(v^2x)$. As described in the beginning of Section \ref{sec5.1},  $\h_0$ and $\til\h^{(v)}_0$ (for each $v\ge 1$) can be viewed as data from \st \ for some $\stp=(\af,\frac12,\bet) \in \R\times \R_{>0}^2$. For convenience set
	\begin{align}\label{def:w}
		w:=\begin{cases}
			\frac43 & \h_0 \mbox{ is non random} \\
			\frac23 & \h_0 \mbox{ is Brownian}
		\end{cases}.
	\end{align}
	To prove \eqref{eq:main3} and \eqref{eq:main4} it suffices to show
	\begin{align*}
		\limsup_{t\to \infty} \frac{\h_t(0)}{t^{1/3}(\log\log t)^{2/3}}\le w^{-2/3}, \qquad \limsup_{t\to \infty} \frac{\h_t(0)}{t^{1/3}(\log\log t)^{2/3}}\ge w^{-2/3}.
	\end{align*}
	The argument now follows by demonstrating the above two bounds, the upper bound and lower bound, separately. This is done in Section  \ref{sec:ublil} and Section \ref{sec:lblil} respectively.

	\subsubsection{Upper bound}  \label{sec:ublil}
	
	\medskip
	
	Fix any $\e>0$. Get $\rho(\stp)>1$ and $\Con(\stp)>0$ from Proposition \ref{pp:weakc}. Get $\til{\rho}(\e,\stp)>1$ so that $\frac1\Con(\til{\rho}-1)^{-\frac13} \e^{3/2}=\frac43$. Set $t_n=\til{\rho}^n$. Let $\gamma$ be such that
	\begin{align}\label{eq:gamma}
		(\tfrac1w+\gamma)(1-\e)(w-\e)>1.
	\end{align}
	
	We claim that
	\begin{align}\label{eq:bcsum1}
		\sum_{n=1}^{\infty}\Pr\left(\mathsf{A}_n^{\gamma}\right)<\infty, \ \ \ \mbox{ where }\mathsf{A}_n^{\gamma}:=\left\{\sup_{t\in [t_n,t_{n+1}]}\frac{\h_t(0)}{t^{1/3}(\log\log t)^{2/3}} \ge (\tfrac1w+\gamma)^{\frac23}\right\}.
	\end{align}
	Clearly one can take $\gamma$ arbitrarily close to $0$ by taking $\e$ arbitrarily close to $0$. Thus by Borel-Cantelli Lemma \eqref{eq:bcsum1} proves the upper bound. So, it suffices to show \eqref{eq:bcsum1}. Towards this end, for all large enough $n$ we define
	\begin{align}
		\label{def:rn}
		u_n:=(\tfrac1w+\gamma)^{2/3}(\log n+\log\log \til\rho)^{2/3}.
	\end{align}
	Observe that 
	\begin{align*}
		\sup_{t\in [t_n,t_{n+1}]}\frac{\h_t(0)}{t^{1/3}(\log\log t)^{2/3}(\tfrac1w+\gamma)^{\frac23}}  \le u_n^{-1}\left[t_n^{-\frac13}\h_{t_n}(0) + \sup_{t \in [t_n, t_{n+1}]} t_n^{-\frac13}|\h_t (0) - \h_{t_n} (0)|\right].
	\end{align*}
	By the union bound,
	\begin{align}
		\Pr\left(\mathsf{A}_n^{\gamma}\right) 
		\label{e.t.ltlilg3}
		&\leq  \P\Big(t_n^{-\frac13}\h_{t_n} (0) \geq (1-\e) u_n \Big) + \P\Big(\sup_{t \in [t_n, t_{n+1}]} t_n^{\frac13}|\h_t (0) - \h_{t_n} (0)| \geq \e u_n \Big).
	\end{align}
	For the first term, by \eqref{eq:tbdg0} and \eqref{eq:tbdg}, we know that for all large enough $n$ we have
	\begin{equation*}
		\P\Big(t_n^{-\frac13}\h_{t_n} (0) \geq (1-\e) u_n \Big) \leq   \exp\big(-(w -\e) (1-\e)^{\frac{3}{2}} u_n^{\frac{3}{2}} \big). 
	\end{equation*}
	As a function in time,
	$\h_t(0)$ is same in distribution as $t_n^{1/3}\h_{t\cdot t_n^{-1}}(0;\til\h_0^{(v)})$ with $v=t_n^{\frac13}$. Here $\h_s(0;\til\h_0^{(v)})$ is the KPZ fixed point started with the initial data $\til\h_0^{(v)}(x)=v^{-1}\h_0(v^2x)$. The choice of $\rho$ and Proposition \ref{pp:weakc} allow us to conclude
	\begin{align*}
		\P\Big(\sup_{t \in [t_n, t_{n+1}]} t_n^{-\frac13}|\h_t (0) - \h_{t_n} (0)| \geq \e u_n \Big) & \leq \mathbb{P}\Big(\sup_{t \in [1, \til{\rho}]} |\h_t (0; \til\h_0^{(v)}) - \h_{1} (0; \til\h_0^{(v)})| \geq \e u_n\Big) \\ & \leq \Con\exp(-\tfrac{1}\Con (\til{\rho}-1)^{-\frac13} \e^{3/2}u_n^{\frac{3}{2}} )= \Con\exp(-\tfrac{4}3u_n^{\frac{3}{2}} ),
	\end{align*}
	where the last equality follows from our choice of $\til{\rho}$. The estimates in the above two math displays are summable in $n$ by the definition of $u_n$ from \eqref{def:rn} and choice of $\gamma$ from \eqref{eq:gamma}. In view of the union bound in \eqref{e.t.ltlilg3}, this proves \eqref{eq:bcsum1}.	
	
	\subsubsection{Lower Bound}\label{sec:lblil} For each $n\in \Z_{>0}$ set $\mathcal{I}_n := [\exp(e^n), \exp(e^{n+1})]$. Take $\e>0$ and consider $\gamma=(1/w-\e)^{2/3}$ where $w$ is defined in \eqref{def:w}. Note that by Borel-Cantelli Lemma it is enough to show
	\begin{align}\label{eq:bcbm}
		\sum_{n = 1}^\infty \mathbb{P}\bigg(\frac{1}{n^{\frac{2}{3}}}\sup_{t \in \mathcal{I}_n} \bigg(\frac{\h_t(0)}{ t^{\frac{1}{3}}}
		\bigg) \leq \gamma\bigg) < \infty.
	\end{align}
	Set $\theta=\frac{\e}{12}>0$. Fix any $n\in \Z_{>0}$ large enough. Let $r_i = e^n + i e^{\theta n}$ for $i\ge 1$ and set 
	\begin{align}\label{eq:delK}
		\nt{i} := e^{r_i} - e^{r_{i-1}}, \quad \mbox{ for } i\in K_n(\theta):=[0,(e-1) e^{(1-\theta)  n}]\cap \Z.
	\end{align}
	Here $e^{r_{-1}}\equiv 0$. Recall $H^{t_2\downarrow t_1}$ defined in \eqref{long}. By union bound we have \begin{align}
		\nonumber	& \mathbb{P}\bigg(\frac{1}{n^{\frac{2}{3}}}\sup_{t \in \mathcal{I}_n} \bigg(\frac{\h_t(0)}{ t^{\frac{1}{3}}}
		\bigg) \leq \gamma\bigg)  \\ \nonumber & \le \P\left(\sup_{i\in K_n(\theta)} \frac{1}{n^{\frac{2}{3}}}\Big(\frac{\h_{e^{r_i}}(0)}{e^{r_i/3}}\Big) \leq \gamma\right)	\\ \nonumber &\leq \mathbb{P}\Big(\sup_{i\in K_n(\theta)} \frac{H^{e^{r_i}\downarrow e^{r_{i-1}}}}{e^{r_i/3}} \leq \gamma n^{2/3}+1 \Big) +\sum_{i\in K_n(\theta)} \Pr\left(\left|\frac{\h_{e^{r_i}}(0)-H^{e^{r_i}\downarrow e^{r_{i-1}}}}{e^{r_i/3}}\right|\ge 1\right)  \\
		\label{eq:bcbound1}
		&\leq \mathbb{P}\Big(\sup_{i\in K_n(\theta)} Y_i \leq \gamma n^{\frac{2}{3}} + 2\Big) + \Con \exp\Big(-\tfrac1{\Con}e^{\tfrac{1}{3} e^{n\theta}}\Big) e^{n(1-\theta)}.
	\end{align}
	The last estimate above follows from Theorem \ref{ind} taking $x\mapsto 1, a\mapsto e^{r_{i-1}}, t \mapsto e^{r_i-r_{i-1}}-1$. Clearly the last term is summable in $n$. Here
	\begin{align*}
		Y_i  & = (\nt{i})^{-\frac{1}{3}}\sup_z \Big(\h_0(z) + \mathcal{L}(z, e^{t_{i-1}}; 0, e^{t_i})\Big)  = \sup_z\bigg(\frac{\h_0((\nt{i})^{\frac{2}{3}} z)}{(\nt{i})^{\frac{1}{3}}} + \mathcal{A}_i(z) - z^2\bigg) 
	\end{align*}
	where $\{\mathcal{A}_i(\cdot)\}_{i\in K_n(\theta)}$ are independent stationary $\text{Airy}_2$ processes. To estimate the first term in r.h.s.~of \eqref{eq:bcbound1}, we now consider two cases.
	
	\medskip
	
	\noindent\textbf{Case 1. $\h_0$ is non-random} Here $Y_i$'s are independent. Then for all large enough $n$ we may bound the first term in r.h.s~of \eqref{eq:bcbound1} as follows.
	\begin{align}\notag
		\mathbb{P}\Big(\sup_{i\in K_n(\theta)} Y_i \leq \gamma n^{\frac{2}{3}} + 2\Big)  & \le \prod_{i\in K_n(\theta)}\Pr\left(Y_i  \le \gamma n^{2/3}+2\right) \\ \notag & \le  \prod_{i\in K_n(\theta)}\Pr\left(\frac{\h_0(0)}{(\nt{i})^{\frac13}}+\mathcal{A}_i(0) \le \gamma n^{2/3}+2\right) \\ & \le \prod_{i\in K_n(\theta)}\Pr\left(\tw  \le \gamma n^{2/3}+3\right), \label{eq:line1}
	\end{align}
	where in the last line we use the fact that $|\h_0(0)| \le \sigma \le (\nt{i})^{\frac13}$ for all large enough $n$. Using precise upper tail bounds for Tracy-Widom distribution \cite{ramirez2011beta} we see that
	\begin{align*}
		\mbox{r.h.s.~of \eqref{eq:line1}} \le (1-\Con \exp(-(\tfrac43+\e)\gamma^{3/2}n))^{(e-1)e^{n(1-\theta)}} \le \exp(-\Con(e-1)e^{n(1-\theta)-n(1-\tfrac{7\e}{12})})
	\end{align*}
	where the last line follows by recalling that $\gamma^{3/2}=(3/4-\e)$ in this case. As $\theta=\frac{\e}{12}$ we see that the above estimate is summable $n$ which in turns shows \eqref{eq:bcbm} for this case.

	\medskip

	\noindent\textbf{Case 2. $\h_0$ is Brownian}
	For $\h_0(z)=\bm(z)$, $Y_i$'s are no longer independent. However we may set $\zeta = \exp(\frac{1}{3} e^{n\theta})$, and define 
	\begin{align}
		\label{eq:brsc}
		Y'_{i} := \sup_{z \in [\zeta^{-1}, \zeta]} \bigg(\frac{\bm((\nt{i})^{\frac{2}{3}} z)}{(\nt{i})^{\frac{1}{3}}} + \mathcal{A}_i(z)-z^2\bigg)
	\end{align}
	and
	\begin{align*}
		Y_i'' := \sup_{z \in [\zeta^{-1}, \zeta]} \bigg(\frac{\bm((\nt{i})^{\frac{2}{3}} z) - \bm((\nt{i})^{\frac{2}{3}} \zeta^{-1})}{(\nt{i})^{\frac{1}{3}}} + \mathcal{A}_i(z) -z^2\bigg).
	\end{align*}
	Note that $Y_i'\le Y_i$ and $Y_i''= Y_i' - \frac{\bm((\nt{i})^{\frac{2}{3}} \zeta^{-1})}{(\nt{i})^{\frac{1}{3}}}$.
	Note that the intervals $[\zeta^{-1} (\nt{i})^{\frac{2}{3}}, \zeta (\nt{i})^{\frac{2}{3}}]$ and  $[\zeta^{-1} (\nt{i+1})^{\frac{2}{3}}, \zeta (\nt{i+1})^{\frac{2}{3}}]$ do not overlap. This ensures $\{Y_i''\}_{i \in K_n(\theta)}$ are independent. Also notice that 
	
	Consequently by union bound, 
	\begin{align*}
		\mathbb{P}\bigg(\sup_{i \in K_n(\theta)} Y_i \leq  n^\frac{2}{3} \gamma + 1 \bigg) &\leq \mathbb{P}\bigg(\sup_{i \in K_n(\theta) } Y'_i \leq n^{\frac{2}{3}} \gamma + 1\bigg)
		\\
		&\leq \mathbb{P}\bigg(\sup_{i \in K_n(\theta)} Y_i'' \leq n^{\frac{2}{3}} \gamma + 2\bigg) + \sum_{i \in K_n(\theta)} \mathbb{P}\Big(|Y_i' - Y_i''| \geq 1\Big) \\ & \leq \prod_{i\in K_n(\theta)}\mathbb{P}\bigg(Y_i'' \leq n^{\frac{2}{3}} \gamma + 2\bigg)+\Con e^{n(1-\theta)} \exp(-\tfrac1{\Con}\zeta),
	\end{align*}
	where the last line follows from independence of $Y_i''$'s and the fact that 
	\begin{equation*}
		\mathbb{P}\Big(|Y_i' - Y_i''| \geq 1\Big) = \mathbb{P}\bigg(\Big|\frac{\bm(\nt{i})^{\frac{2}{3}} \zeta^{-1})}{(\nt{i})^{\frac{1}{3}}}\Big| \geq 1\bigg) \leq \Con \exp(-\tfrac1{\Con}\zeta).
	\end{equation*}
	By union bound and Lemma \ref{lem:onepttail} \ref{2.6d} we see that for $n$ large enough
	\begin{align*}
		\mathbb{P}\Big(Y_i'' \leq n^{\frac{2}{3}}\gamma + 2\Big) &\leq \mathbb{P}\Big(\sup_{z \in [\zeta^{-1}, \zeta]}\Big(\frac{\bm((\nt{i})^{\frac{2}{3}} z)}{(\nt{i})^{\frac{1}{3}}} + \mathcal{A}_i (z) - z^2\Big) \leq n^{\frac{2}{3}}\gamma +3\Big) 
		+ \mathbb{P}\bigg(\frac{|\bm((\nt{i})^{\frac{2}{3}} \zeta^{-1})|}{(\nt{i})^{\frac{1}{3}}} \geq 1\bigg)\\
		&\leq 1 - \exp\big(-(\tfrac{2+\e}{3}) n \gamma^{\frac{3}{2}}\big) + \Con\exp(-\tfrac1{\Con}\zeta) \\ & \le 1 - e^{-(1-\frac{\e}{6}) n} + \Con\exp(-\tfrac1{\Con}\zeta)
	\end{align*}
	where in the last line we use the fact that $\gamma^{3/2}=(3/2-\e)$ in this case. Consequently, we have 
	\begin{align*}
		\prod_{i \in K_n(\theta)} \mathbb{P}\Big(Y_i'' \leq n^{\frac{2}{3}} \gamma + 2\Big)
		&\leq \Big(1 - \tfrac{1}{\Con}e^{-(1-\frac{\e}{6}) n} + \Con\exp(-\tfrac1\Con e^{\frac{1}{3} e^{n\theta}})\Big)^{(e-1)e^{n(1-\theta)}}\\
		&\leq \exp\Big(-(e-1)e^{n(1-\theta)} \Big(e^{-(1-\frac{\e}6) n} - \Con\exp(-\tfrac1\Con e^{\frac{1}{3} e^{n\theta}})
		\Big)\Big)
	\end{align*}
	which is summable in $n$ for $\theta=\frac{\e}{12}$. This concludes the proof of \eqref{eq:bcbm}.

	\subsection{Short time LIL: Proof of Theorem \ref{thm:main}} \label{sec6.2}
	
	We now turn towards the proof of the short time LIL: Theorem \ref{thm:main}. As with the proof of Theorem \ref{thm:longLILa} and \ref{thm:longLILb}, it is enough to show
	\begin{equation}\label{eq:sbdd}
		\limsup_{t \downarrow 1} \frac{\h_t (0) - \h_{1} (0)}{(t-1)^{\frac{1}{3}} (\log \log\frac1{t-1})^{\frac{2}{3}}} \leq (\tfrac{3}{2})^{\frac{2}{3}}, \qquad \limsup_{t \downarrow 1} \frac{\h_t (0) - \h_{1} (0)}{(t-1)^{\frac{1}{3}} (\log \log\frac1{t-1})^{\frac{2}{3}}} \geq (\tfrac{3}{2})^{\frac{2}{3}}.
	\end{equation}
	We now proceed to show the above lower bound and upper bound in Section \ref{sec:subld} and Section \ref{sec:slwld} respectively.
	
	\subsubsection{Upper Bound}\label{sec:subld} Fix  $\delta>0$. For simplicity let us write $\h_t:=\h_t(0)$. Fix $\rho\in (0,1)$ to be chosen appropriately in terms of $\delta$ later. Consider $s_n = 1 + \rho^n$. Note that for all large enough $n$ 
	\begin{align}
		\notag
		\sup_{t \in [s_{n+1}, s_n]}\frac{|\h_t - \h_{1}|}{(t-1)^{\frac{1}{3}}} &\leq \frac{|\h_{s_n} - \h_{1}|}{(s_{n+1} - 1)^{\frac{1}{3}}} + \sup_{t \in [s_{n+1}, s_n]} \frac{|\h_t- \h_{s_n} |}{(s_{n+1}-1)^{\frac{1}{3}}}\\
		&= \rho^{-\frac{1}{3}} \frac{|\h_{s_n} - \h_{1}|}{(s_{n} - 1)^{\frac{1}{3}}} + (\rho^{-1} - 1)^{\frac{1}{3}} \sup_{t \in [s_{n+1}, s_n]} \frac{|\h_t - \h_{s_n}|}{(s_{n}-s_{n+1})^{\frac{1}{3}}}. \notag
	\end{align}
	By union bound we have 
	\begin{align*}
		& \P\left(\sup_{t \in [s_{n+1}, s_n]}\frac{|\h_t- \h_{1}|}{(t-1)^{\frac{1}{3}} (\log \log \frac{1}{t-1})^{\frac{2}{3}}} \geq (\tfrac{3}{2})^{\frac{2}{3}} + \delta\right) \\ & \leq \P\left(\sup_{t \in [s_{n+1}, s_n]}\frac{|\h_t - \h_{1}|}{(t-1)^{\frac{1}{3}}} \geq \big((\tfrac{3}{2})^{\frac{2}{3}} + \delta) \log^{\frac23} n\right) \\
		&\leq \P\left(\rho^{-\frac{1}{3}} \frac{|\h_{s_n} - \h_{1}|}{(s_{n} - 1)^{\frac{1}{3}}} \geq \big((\tfrac{3}{2})^{\frac{2}{3}} + \tfrac{\delta}{2}\big) \log^{\frac23} n\right)\\
		& \hspace{2cm}  + \P\left((\rho^{-1} - 1)^{\frac{1}{3}} \sup_{t \in [s_{n+1}, s_n]} \frac{|\h_t - \h_{s_n} |}{(s_{n}-s_{n+1})^{\frac{1}{3}}} \geq \tfrac{\delta}{2} \log^{\frac23} n\right).
	\end{align*}
	We apply Proposition \ref{prop:stincrementtail} with $\e=\frac{\delta}{100}$ and get
	\begin{equation}\label{eq:temp11}
		\P\left(\rho^{-\frac{1}{3}} \frac{|\h_{s_n} - \h_{1}|}{(s_{n} - 1)^{\frac{1}{3}}} \geq ((\tfrac{3}{2})^{\frac{2}{3}} + \tfrac{\delta}{2})\log^{\frac23} n\right) \leq \Con \exp\Big(-\rho^{1/2}(\tfrac23-\tfrac{\delta}{100})((\tfrac{3}{2})^{\frac{2}{3}} + \tfrac{\delta}{2})^{3/2}\log n\Big).
	\end{equation}
	On the other hand, by Proposition \ref{pp:weakc} we have
	\begin{equation}\label{eq:temp12}
		\P\left((\rho^{-1} - 1)^{\frac{1}{3}} \sup_{t \in [s_{n+1}, s_n]} \frac{|\h_t- \h_{s_n}|}{(s_{n}-s_{n+1})^{\frac{1}{3}}} \geq \tfrac{\delta}{2} \log^{\frac23} n\right) \leq \Con\exp\left(-\tfrac1\Con \frac{{\delta}^{3/2}}{(\rho^{-1}-1)^{1/2}2^{3/2}}\log n\right).
	\end{equation}
	Now one can choose $\rho=\rho(\delta)$ close to $1$ but less than $1$ so that the coefficient of $\log n$ on r.h.s.~of \eqref{eq:temp11} and \eqref{eq:temp12} is strictly larger than $1$. This forces both the estimates on \eqref{eq:temp11} and \eqref{eq:temp12} summable in $n$. Thus for this choice of $\rho$ we have
	\begin{align*}
		\sum_{n=1}^{\infty} \P\Big(\sup_{t \in [s_{n+1}, s_n]}\frac{|\h_t- \h_{1}|}{(t-1)^{\frac{1}{3}} (\log \log \frac1{t-1})^{\frac{2}{3}}} \geq (\tfrac{3}{2})^{\frac{2}{3}} + \delta\Big) < \infty.
	\end{align*}
	By Borell Cantelli Lemma  with probability $1$
	\begin{equation*}
		\limsup_{t\downarrow 1}\frac{|\h_t-\h_{1}|}{(t-1)^{\frac{1}{3}} (\log \log (t-1)^{-1})^{\frac{2}{3}}} \leq (\tfrac{3}{2})^{\frac{2}{3}} + \delta.
	\end{equation*}
	Letting $\delta \downarrow 0$ concludes the upper bound in \eqref{eq:sbdd}.
	
	\subsubsection{Lower Bound}\label{sec:slwld} For each $n\in \Z_{>0}$ set $\mathcal{I}_n := [\exp(e^n), \exp(e^{n+1})]$. Take $\e>0$ and consider $\gamma=\left(\frac32-\e\right)^{\frac23}$. Note that by Borel-Cantelli Lemma it is enough to show
	\begin{align}\label{eq:slb}
		\sum_{n=1}^{\infty} \Pr\left(\frac1{n^{\frac23}}\sup_{t\in\mathcal{I}_n}\frac{\h_{1+t^{-1}}(0)-\h_1(0)}{t^{-\frac13}} \le \gamma \right) <\infty.
	\end{align}
	Set $\theta=\frac{\e}{12}>0$. Fix any $n\in \Z_{>0}$ large enough. Let $r_i = e^n +i e^{\theta n}$ for $i\ge 1$ and set \begin{align*}
		\ft{i} := e^{-r_{i-1}} - e^{-r_{i}}, \quad \mbox{ for } i\in K_n(\theta):=[0,(e-1) e^{(1-\theta)  n}]\cap \Z.
	\end{align*}
	Here $e^{-r_{-1}}\equiv 0$. Recall $H_1^{t_2\downarrow t_1}$ defined in \eqref{def:dwnar}. Observe that
	\begin{align} \nonumber
		\Pr\left(\frac1{n^{\frac23}}\sup_{t\in\mathcal{I}_n}\frac{\h_{1+t^{-1}}(0)-\h_1(0)}{t^{-\frac13}} \le \gamma \right)  & \le \Pr\left(\sup_{i\in K_n(\theta)}\frac{\h_{1+e^{-r_i}}(0)-\h_1(0)}{e^{-r_i/3}} \le \gamma n^{2/3} \right) \\ \label{t1} & \le \Pr\left(\sup_{i\in K_n(\theta)}\frac{H_{1}^{e^{-r_i}\downarrow e^{-r_{i+1}}}	-\h_1(0)}{e^{-r_i/3}} \le \gamma n^{2/3}+1 \right)\\ & \hspace{1cm}+\sum_{i\in K_n(\theta)}\Pr\left(\left|\frac{\h_{1+e^{-r_i}}(0)-H_{1}^{e^{-r_i}\downarrow e^{-r_{i+1}}}}{e^{-r_i/3}}\right|\ge 1\right). \label{t2}
	\end{align}
	Temporarily set $t:=e^{r_{i+1}-r_i}-1=e^{e^{\theta n}}-1$, $a^{-1}:=e^{-r_i}-e^{-r_{i+1}} \le e^{-r_i}$ which implies $a \ge e^{r_i} \ge e^{e^n}$. Clearly for large enough $n$, $(a,t)$ is a permissible pair. Thus for the term in \eqref{t2}, by Theorem \ref{thm:lr} we get that
	\begin{align*}
		\sum_{i\in K_n(\theta)}\Pr\left(\left|\frac{\h_{1+e^{-r_i}}(0)-H_{1}^{e^{-r_i}\downarrow e^{-r_{i+1}}}	}{e^{-r_i/3}}\right|\ge 1\right) & \le  \Con e^{(1-\theta)n}e^{\frac23e^{n+1}}\exp\left(-\tfrac1\Con e^{\frac1{16}e^{\theta n}}\right) \\ & = \Con \exp\left((1-\theta)n+\tfrac23e^{n+1}-\tfrac1{\Con}e^{\frac1{16}e^{\theta n}}\right)
	\end{align*}
	which is summable over $n$. For the term in \eqref{t1} note that
	\begin{align*}
		\frac{H_{1}^{e^{-r_i}\downarrow e^{-r_{i+1}}}	-\h_1(0)}{e^{-r_i/3}}  =  (1-\exp(-e^{\theta n}))^{\frac13} \cdot\sup_{y\in \R} \left[\frac{\h_1((\ft{i})^{2/3}y)-\h_1(0)}{\ft{i}^{1/3}}+\calA_i(y)-y^2\right]  
	\end{align*}
	where $\calA_i$ are independent stationary $\operatorname{Airy}_2$ processes independent of $\h_1(\cdot)$. Set $\zeta=\exp(\frac13e^{\theta n})$. Observe that $\ft{i}^{2/3}\zeta= \ft{{i-1}}^{2/3}\zeta^{-1}$. Thus $\{(\ft{i}^{2/3}\zeta^{-1},\ft{i}^{2/3}\zeta)\}_i$ are disjoint intervals. Set $a_n:=\ft{0}^{2/3}\zeta \le \exp(-\frac23r_0+\frac13e^{\theta n})$ and $\mu_n:=\exp(\frac16r_0)\zeta^{-1}=\exp(\frac16r_0-\frac13e^{\theta n})$. Define
	\begin{align*}
		Y_i & := \sup_{y\in[\zeta^{-1},\zeta]} \left[\frac{\h_1((\ft{i})^{\frac23}y)-\h_1(0)}{\ft{i}^{1/3}}+\calA_i(y)-y^2\right], \\ 
		Y_i^{\operatorname{drift}} &  := \sup_{y\in[\zeta^{-1},\zeta]} \left[\frac{\h_1^{-\mu_n}((\ft{i})^{\frac23}y)-\h_1^{-\mu_n}(0)}{\ft{i}^{1/3}}+\calA_i(y)-y^2\right].
	\end{align*}
	Recall the event $\EE_1^{\mu_n}(a_n)$ from \eqref{eq:etma}.	Observe that by Lemma \ref{lem:argmaxprop} \ref{lem:compare} we have
	\begin{align*}
		\eqref{t1}  \le \Pr\left(\sup_{i\in K_n(\theta)}Y_i  \le \gamma n^{2/3}+2\right) & \le \Pr\left(\sup_{i\in K_n(\theta)}Y_i  \le \gamma n^{2/3}+2, \EE_1^{\mu_n}(a_n)\right)+\Pr\left(\EE_1^{\mu_n}(a_n)^c\right) \\ & \le \Pr\left(\sup_{i\in K_n(\theta)}Y_i^{\operatorname{drift}}  \le \gamma n^{2/3}+2\right)+\Pr\left(\EE_1^{\mu_n}(a_n)^c\right).
	\end{align*}
	We claim that 
	\begin{align}
		\label{eq:claim4}
		\Pr\left(\m{E}_1^{a_n}(\mu_n)^c\right) \le \Con\exp\left(-\tfrac1\Con \mu_n^{3/2}\right)=\Con\exp\left(-\tfrac1{\Con}\exp(\tfrac14r_0-\tfrac12e^{\theta n})\right),
	\end{align}
	which is summable over $n$. We will prove \eqref{eq:claim4} in a moment. Let us first finish the proof of the lower bound assuming it.

	Note that as a process in $x$, by Lemma \ref{lem:2.3} \ref{2.3c} we have $\h_1^{-\mu_n}(x)-\h_1^{-\mu_n}(0)\stackrel{d}{=}\B(x)-\mu_n x$. As
	\begin{align*}
		|\mu_n(\nt{i})^{1/3}\zeta| = (e^{\frac12r_0-r_{i-1}}-e^{\frac12r_0-r_{i}})^{1/3} \le e^{-\frac12r_0} \le 1.
	\end{align*}
	Thus $Y_i^{\operatorname{drift}}\ge Y_i'-1$ where 
	\begin{align*}
		Y'_{i} := \sup_{z \in [\zeta^{-1}, \zeta]} \bigg(\frac{\bm((\ft{i})^{\frac{2}{3}} z)}{(\ft{i})^{\frac{1}{3}}} + \mathcal{A}_i(z)-z^2\bigg).
	\end{align*}
	Repeating the arguments of \textbf{Case 2} of proof of long-time LIL we conclude \begin{align}\label{br-tail}
		\sum_{n=1}^{\infty}\Pr\left(\sup_{i\in K_n(\theta)}Y_i'  \le \gamma n^{2/3}+3\right) <\infty.
	\end{align}
	Thus the term in \eqref{t1} is summable in $n$. This establishes \eqref{eq:slb} concluding the proof modulo \eqref{eq:claim4}. Let us now justify \eqref{eq:claim4}. Observe that for any $a \in [-1,1]$ and $\mu >0$ by the same argument as in \eqref{e.l.maxbound4} we have 
	\begin{equation*}
		\begin{aligned}
			\P\big(\EE_1^{\mu}
			(a)^c\big) & \leq \mathbb{P}\Big(Z^{0}_1(0)  \leq a-\tfrac14\mu\Big) +  \mathbb{P}\Big(Z_1 (a; \h_0)\geq \tfrac{1}{4}\mu\Big) \\
			& \hspace{2cm}+\mathbb{P}\Big(Z^{0}_1(0) \geq \tfrac14\mu-a\Big) +  \mathbb{P}\Big(Z_1 (-a; \h_0)\leq -\tfrac{1}{4}\mu\Big).
		\end{aligned}
	\end{equation*}
	Applying Corollary \ref{lem:exitbound} we see that the r.h.s.~of above equation is at most $\Con\exp(-\frac1\Con \mu^3)$ where the $\Con>0$ depends only on $\stp$ and is free of $a$ and $\mu$. As $a_n\in [-1,1]$ for all large enough $n$, this proves \eqref{eq:claim4}.

	\appendix
	\section{Technical Lemmas}\label{app}
	
	In this section we prove two technical lemmas that were used in the main sections. The first is Lemma \ref{b_bb} which gives a technical estimate related to Brownian motion. We recall the statement of the result for reader's convenience. 
	
	\begin{lemma}[Lemma \ref{b_bb}]   There exist universal constants $\Con>0$ such that for all $r,s>0$ we have
		\begin{align}\label{eq:b_bb1}
			\Pr\Big(\sup_{y\in \R, |z|\le r} \left[ \B(y)-\B(z)-\tfrac{(y-z)^2}{4}\right] \ge s\Big) \le (r+1)\cdot \se{3/2}
		\end{align}
		where $\B(x)$ is a two-sided Brownian motion with diffusion coefficient $2$.
	\end{lemma}
	
	\begin{proof} Fix $k_1,k_2 \in \Z$. Define
		\begin{align*}
			m_{k_1,k_2}:=\inf\limits_{y\in [k_1,k_1+1],z\in [k_2,k_2+1]} |y-z|, \quad \mathsf{A}_{k_1,k_2} & := \Big\{\sup_{y\in [k_1,k_1+1], z\in [k_2,k_2+1]} \big[ \B(y)-\B(z)-\tfrac{(y-z)^2}{4}\big] \ge s\Big\}.
		\end{align*}
		Let us set $ 2\sup_{x\in [0,1]}\sqrt{x\log\frac{2}{x}}=:\tau<\infty$ We consider the following events
		\begin{align*}
			\mathsf{B}_{k_1,k_2} & := \Big\{\sup_{y\in [k_1,k_1+1],z\in [k_2,k_2+1]} \big|\B(y)-\B(z)-\B({k_1})+\B({k_2})\big|\ge \tfrac{s}{2}+\tfrac{m_{k_1,k_2}^2}{8}\Big\},\\
			\mathsf{C}_{k_1,k_2} & := \Big\{|\B({k_1})-\B({k_2})|\ge \tfrac{s}{2}+\tfrac{m_{k_1,k_2}^2}{8} \Big\}.
		\end{align*}
		Define the 2D random process $\mathfrak{F}(y,z):=\B(y)-\B(z)$. Note that by tail estimates of Brownian motion we have
		\begin{align*}
			\Pr(|\mathfrak{F}(y,z+h)-\mathfrak{F}(y,z)|\ge u\sqrt{h})\le \Con e^{-\frac{u^2}{\Con}}, \quad 	\Pr(|\mathfrak{F}(y+h,z)-\mathfrak{F}(y,z)|\ge u\sqrt{h}) \le \Con e^{-\frac{u^2}{\Con}}.
		\end{align*}
		Thus by Lemma 3.3 in \cite{dv21}, we get that
		\begin{align*}
			\Pr\Big(\sup_{y\in [k_1,k_1+1],z\in [k_2,k_2+1]} \Big|\frac{\mathfrak{F}(y,z)-\mathfrak{F}(k_1,k_2)}{\sqrt{(z-k_2)\log\frac2{z-k_2}}+\sqrt{(y-k_1)\log\frac2{y-k_1}}}\Big|\ge u\Big) \le \Con e^{-\frac{u^2}{\Con}}.
		\end{align*}
		We can set $u=\frac1{\tau}\big(\frac{s}2+\frac{m_{k_1,k_2}^2}{8}\big)$ and adjust the constants to get 
		\begin{align*}
			\Pr\big(\mathsf{B}_{k_1,k_2}\big)\le \Con\exp\Big(-\tfrac1\Con(4s+m_{k_1,k_2}^2)^2\Big).
		\end{align*}
		On the other hand when $k_1=k_2$, the event $\mathsf{C}_{k_1,k_2}$ has probability zero. For $k_1\neq k_2$ by stationary increment properties of Brownian motion we have
		\begin{align*}
			\Pr\Big(|\B({k_1})-\B({k_2})|\ge \tfrac{s}{2}+\tfrac{m_{k_1,k_2}^2}{8}\Big)\le \Con\exp\Big(-\tfrac1\Con\frac{(4s+m_{k_1,k_2}^2)^2}{|k_1-k_2|}\Big).
		\end{align*}
		Since $\mathsf{A}_{k_1,k_2} \subset \mathsf{B}_{k_1,k_2}\cup \mathsf{C}_{k_1,k_2}$, by union bound we get 
		\begin{align*}
			\Pr(\mathsf{A}_{k_1,k_2}) \le \Con\exp\big(-\tfrac1\Con(4s+m_{k_1,k_2}^2)^2\big)+\Con\exp\Big(-\frac1\Con\frac{(4s+m_{k_1,k_2}^2)^2}{|k_1-k_2|}\Big)\ind_{k_1\neq k_2}
		\end{align*}
		Fixing $k_2$, we first take the sum over $k_1\in \Z$. Since $m_{k_1,k_2} \ge \frac12|k_1-k_2|$ whenever $|k_1-k_2|\ge 2$, the first term sums to $c_1\exp(-c_2s^2)$. For the second term we approximate the sum as an integral. Indeed,
		\begin{align*}
			\sum_{k_1\in \Z} \Con\exp\Big(-\tfrac1\Con\frac{(4s+m_{k_1,k_2}^2)^2}{|k_1-k_2|}\Big)\ind_{k_1\neq k_2} \le \Con\int_{\R}\exp\left(-\tfrac1\Con\frac{(s+|x|^2)^2}{|x|}\right)\d x \le \Con\exp(-\tfrac1\Con s^{3/2}).
		\end{align*}
		The last inequality follows by noting that the exponent attains maximum of $-c_2s^{3/2}$ when $|x|=\sqrt{s}$. Thus by union bound again,
		\begin{align*}
			\mbox{l.h.s.~of \eqref{eq:b_bb1}} \le \sum_{k_2=-\lceil r\rceil}^{\lceil r\rceil}\sum_{k_1\in \Z} \P(\mathsf{A}_{k_1,k_2}) \le (r+1)\cdot \se{3/2}.
		\end{align*} 
		This completes the proof. 
	\end{proof}
	
	We next collect an elementary real analysis inequality that is used in defining permissible pairs in Definition \ref{pair}.
	
	\begin{lemma}[Real Analysis Inequality] \label{l:rai} For each $\alpha>1$, there exists a constant $\Delta_\alpha>0$ such that for all $\gamma\ge 2$ we have
		\begin{align*}
			\sup_{x\in [0,1]} \left[x^{\frac12}\log^\alpha \tfrac1x-\gamma x^2\right] \le \Delta_{\alpha} \gamma^{-\frac13}\log^{4\alpha/3} \gamma.
		\end{align*}
	\end{lemma}
	\begin{proof} We use $\Con>0$ to denote a generic constant dependent only on $\alpha$ changing from line to line. Let $y=x\gamma^{\frac23}$ so that
		\begin{align*}
			\sup_{x\in [0,1]} \left[x^{\frac12}\log^\alpha \tfrac1x-\gamma x^2\right] =
			\sup_{y\in [0,\gamma^{\frac23}]} \gamma^{-\frac13}\left[y^{\frac12}\log^\alpha \tfrac{\gamma^{2/3}}{y}-y^2\right].
		\end{align*}
		Note that for $y\in [0,1]$, 
		\begin{align*}
			y^{\frac12}\log^\alpha \tfrac{\gamma^{2/3}}{y} \le 2^\alpha\left[y^{\frac12}\log^{\alpha} \gamma^{2/3}+y^{\frac12}\log^{\alpha}\tfrac1y\right] \le \Con \log^{\alpha}\gamma.
		\end{align*}
		For $y\in [1,\gamma^{2/3}]$ we have
		\begin{align*}
			y^{\frac12}\log^\alpha \tfrac{\gamma^{2/3}}{y}-y^2 \le y^{\frac12}\log^{\alpha} \gamma^{2/3} - y^2  \le \Con \log^{4\alpha/3} \gamma,
		\end{align*}
		where the last inequality follows from the fact that $\sqrt{y}r-y^2$ is concave in $y$ and attains maximum at $y=(r/4)^{2/3}$. Combining the last two displays we obtain the required inequality.
	\end{proof}

	\bibliographystyle{alpha}
	\bibliography{frt}
\end{document}